\documentclass[12pt,reqno,a4paper]{amsart}            
\usepackage{ifthen}

\usepackage[totalwidth=480pt, totalheight=720pt]{geometry}

\newcommand{\finalVersion}{}  

\usepackage{scrpage2}    
\usepackage{scrtime}     

\usepackage{ifthen}       

\ifthenelse{\isundefined \finalVersion}
{ 
  \rehead{\texttt{\jobname.tex}, \today, \thistime}
  \lohead{\texttt{\jobname.tex}, \today, \thistime}

  \newcommand{\marker}[1]{\fbox{\rule{0pt}{0.1ex}\textbf{#1}}}
  \newcommand{\markerO}{\marker{Olaf}}
  \newcommand{\markerJ}{\marker{Jan}}

  \newcounter{lookcounter}
  \setcounter{lookcounter}{0}
  \newcommand{\look}[2][$\diamond$]
             {
               \stepcounter{lookcounter} 
               \marker{#1}
               \footnote[\arabic{lookcounter}]{\marker{$\diamond$} #2}
             }
  \newcommand{\lookO}[1]
             {
               \stepcounter{lookcounter} 
               \markerO
               \footnote[\arabic{lookcounter}]{\markerO #1}
             }
  \newcommand{\lookJ}[1]
             {
               \stepcounter{lookcounter} 
               \markerJ
               \footnote[\arabic{lookcounter}]{\markerJ #1}
             }

  \ifthenelse{\isundefined \OlafsVersion} 
  {}
  {\usepackage{/home/post/Aktuell/LaTeX/my-showkeys}}  
}
{ 
  \lohead{\textsf{\shorttitle}}
  \rehead{\textsf{\authors}}
  \automark[subsection]{section}
  \automark[subsection]{section}

  \newcommand{\look}[1]{}%
  \newcommand{\lookO}[1]{}%
  \newcommand{\lookJ}[1]{}%
}

\newcommand{\comment}[1]{}   

\usepackage{graphicx}                

\usepackage{amsmath, amssymb}
\usepackage{amsthm}

\usepackage{amscd}

\usepackage{accents}     
\usepackage{mathtools}  

\usepackage{bbm}         

\usepackage{mathrsfs}    

\usepackage{xspace}  

\numberwithin{equation}{section}

\newcounter{myenumi}

\newcommand{\myfont}{\sffamily}

\newcommand{\myparagraph}[1]{\noindent\textbf{\myfont{#1}}}

\newtheoremstyle{mythmstyle}
  {\topsep}
  {\topsep}
  {\itshape}
  {}
  {\bfseries \myfont}
  {.}
  {.5em}
  {}

\newtheoremstyle{mydefstyle}
  {\topsep}
  {\topsep}
  {\normalfont}
  {}
  {\bfseries \myfont}
  {.}
  {.5em}
  {}

\swapnumbers                    

\theoremstyle{mythmstyle}       

\newtheorem{theorem}{Theorem}[section]

\newtheorem{proposition}[theorem]{Proposition}
\newtheorem{lemma}[theorem]{Lemma}
\newtheorem{corollary}[theorem]{Corollary}

\newcounter{intro}

\theoremstyle{mydefstyle}        
\newtheorem{definition}[theorem]{Definition}

\newtheorem{remark}[theorem]{Remark}

\newtheorem*{remark*}{Remark}

\expandafter\let\expandafter\oldproof\csname\string\proof\endcsname
\let\oldendproof\endproof
\renewenvironment{proof}[1][\bfseries\myfont\proofname]{%
  \oldproof[\bfseries \myfont #1]%
}{\oldendproof}

\makeatletter
\renewcommand\section{\@startsection{section}{1}%
  \z@{.7\linespacing\@plus\linespacing}{.5\linespacing}%
  {\Large\myfont\bfseries}}
\renewcommand\subsection{\@startsection{subsection}{2}%
  \z@{-.5\linespacing\@plus-.7\linespacing}{.5\linespacing}%
  {\large\myfont\bfseries}}
\renewcommand\subsubsection{\@startsection{subsubsection}{3}%
  \z@{.5\linespacing\@plus.7\linespacing}{-.5em}%
  {\myfont\bfseries}}
\renewenvironment{abstract}{%
  \ifx\maketitle\relax
    \ClassWarning{\@classname}{Abstract should precede
      \protect\maketitle\space in AMS document classes; reported}%
  \fi
  \global\setbox\abstractbox=\vtop \bgroup
    \normalfont\Small
    \list{}{\labelwidth\z@
      \leftmargin3pc \rightmargin\leftmargin
      \listparindent\normalparindent \itemindent\z@
      \parsep\z@ \@plus\p@
      
    }%
    \item[\hskip\labelsep
      \myfont\bfseries
    \abstractname.]%
}{%
  \endlist\egroup
  \ifx\@setabstract\relax \@setabstracta \fi
}
\renewcommand\contentsnamefont{\myfont\bfseries}
\renewcommand\@starttoc[2]{\begingroup
  \setTrue{#1}%
  \par\removelastskip\vskip\z@skip
  \@startsection{}\@M\z@{\linespacing\@plus\linespacing}%
    {.5\linespacing}{
      \contentsnamefont}{#2}%
  \ifx\contentsname#2%
  \else \addcontentsline{toc}{section}{#2}\fi
  \makeatletter
  \@input{\jobname.#1}%
  \if@filesw
    \@xp\newwrite\csname tf@#1\endcsname
    \immediate\@xp\openout\csname tf@#1\endcsname \jobname.#1\relax
  \fi
  \global\@nobreakfalse \endgroup
  \addvspace{32\p@\@plus14\p@}%
  \let\tableofcontents\relax
}
\renewcommand\@settitle{\begin{center}%
  \baselineskip14\p@\relax
    \LARGE
    \bfseries
    \myfont
  \@title
  \end{center}%
}
\renewcommand\@setauthors{%
  \begingroup
  \def\thanks{\protect\thanks@warning}%
  \trivlist
  \centering\footnotesize \@topsep30\p@\relax
  \advance\@topsep by -\baselineskip
  \item\relax
  \author@andify\authors
  \def\\{\protect\linebreak}%
  \large
  \myfont\bfseries\authors
  \ifx\@empty\contribs
  \else
    ,\penalty-3 \space \@setcontribs
    \@closetoccontribs
  \fi
  \endtrivlist
  \normalfont\myfont\@setaddresses
  \endgroup
}
\renewcommand\@setaddresses{\par
  \nobreak \begingroup
\footnotesize
  \def\author##1{\nobreak\addvspace\bigskipamount}%
  \def\\{\unskip, \ignorespaces}%
  \interlinepenalty\@M
  \def\address##1##2{\begingroup
    \par\addvspace\bigskipamount\indent
    \@ifnotempty{##1}{(\ignorespaces##1\unskip) }%
    {
      \ignorespaces##2}\par\endgroup}%
  \def\curraddr##1##2{\begingroup
    \@ifnotempty{##2}{\nobreak\indent\curraddrname
      \@ifnotempty{##1}{, \ignorespaces##1\unskip}\/:\space
      ##2\par}\endgroup}%
  \def\email##1##2{\begingroup
    \@ifnotempty{##2}{\nobreak\indent\emailaddrname
      \@ifnotempty{##1}{, \ignorespaces##1\unskip}\/:\space
      \ttfamily##2\par}\endgroup}%
  \def\urladdr##1##2{\begingroup
    \def~{\char`\~}%
    \@ifnotempty{##2}{\nobreak\indent\urladdrname
      \@ifnotempty{##1}{, \ignorespaces##1\unskip}\/:\space
      \ttfamily##2\par}\endgroup}%
  \addresses
  \endgroup
}
\renewcommand\enddoc@text{\ifx\@empty\@translators \else\@settranslators\fi
}
\renewcommand\@secnumfont{\myfont\bfseries} 

\renewcommand\maketitle{\par
  \@topnum\z@ 
  \@setcopyright
  \thispagestyle{firstpage}
  \ifx\@empty\shortauthors \let\shortauthors\shorttitle
  \else \andify\shortauthors
  \fi
  \@maketitle@hook
  \begingroup
  \@maketitle
  \toks@\@xp{\shortauthors}\@temptokena\@xp{\shorttitle}%
  \toks4{\def\\{ \ignorespaces}}
  \edef\@tempa{%
    \@nx\markboth{\the\toks4
      \@nx\MakeUppercase{\the\toks@}}{\the\@temptokena}}%
  \@tempa
  \endgroup
  \c@footnote\z@
  \@cleartopmattertags
}
\makeatother

\newcommand{\Sec}[1]{Section~\ref{sec:#1}}

\newcommand{\Subsec}[1]{Subsection~\ref{ssec:#1}}

\newcommand{\Fig}[1]{Figure~\ref{fig:#1}}

\newcommand{\Thm}[1]{Theorem~\ref{thm:#1}}

\newcommand{\Lem}[1]{Lemma~\ref{lem:#1}}

\newcommand{\Cor}[1]{Corollary~\ref{cor:#1}}

\newcommand{\CorS}[2]{Corollaries~\ref{cor:#1}--\ref{cor:#2}}

\newcommand{\Prp}[1]{Proposition~\ref{prp:#1}}

\newcommand{\Rem}[1]{Remark~\ref{rem:#1}}

\newcommand{\Def}[1]{Definition~\ref{def:#1}}
\newcommand{\Defs}[2]{Definitions~\ref{def:#1} and~\ref{def:#2}}

\newcommand{\abs}[2][{}]{\lvert{#2}\rvert_{{#1}}}    
\newcommand{\abssqr}[2][{}]{\lvert{#2}\rvert^2_{#1}} 
\newcommand{\bigabs}[2][{}]{\bigl\lvert{#2}\bigr\rvert_{#1}}     
\newcommand{\bigabssqr}[2][{}]{\bigl\lvert{#2}\bigr\rvert^2_{#1}}
\newcommand{\Bigabssqr}[2][{}]{\Bigl\lvert{#2}\Bigr\rvert^2_{#1}}

\newcommand{\normsymb}{\|}
\newcommand{\bignormsymb}[1]{#1\|}

\newcommand{\norm}[2][{}]{\normsymb{#2}\normsymb_{{#1}}}    
\newcommand{\normsqr}[2][{}]{\normsymb{#2}\normsymb^2_{#1}} 
\newcommand{\bignorm}[2][{}]{\bignormsymb{\bigl}{#2}\bignormsymb{\bigr}_{#1}}
\newcommand{\Bignorm}[2][{}]{\bignormsymb{\Bigl}{#2}\bignormsymb{\Bigr}_{#1}}

\newcommand{\iprod}[3][{}]{\langle{#2},{#3}\rangle_{#1}}  
\newcommand{\bigiprod}[3][{}]{\bigl\langle{#2},{#3}\bigr\rangle_{#1}}

\newcommand{\set}[2]{\{ \, #1 \, | \, #2 \, \} }      
\newcommand{\bigset}[2]{\bigl\{ \, #1 \, \bigl|\bigr. \, #2 \, \bigr\} }

\newcommand{\map}[3]{ #1 \colon #2 \longrightarrow #3}    

\newcommand{\bd}  {\partial}          

\newcommand{\restr}[1]{{\restriction}_{#1}} 

\def\XXint#1#2#3{{\setbox0=\hbox{$#1{#2#3}{\int}$}
     \vcenter{\hbox{$#2#3$}}\kern-.5\wd0}}
\def\XXsum#1#2#3{{\setbox0=\hbox{$#1{#2#3}{\int}$}
     \vcenter{\hbox{$#2#3$}}\kern-.60\wd0}}

\newcommand{\card}[1]{\lvert#1\rvert}   

\newcommand{\dd}    {\, \mathrm d}    

\DeclareMathOperator{\dom}    {dom}
\DeclareMathOperator{\ran}    {ran}
\DeclareMathOperator{\id}     {id}   

\DeclareMathOperator{\supp}   {supp}

\newcommand{\specsymb} {\sigma} 

\newcommand{\spec}[2][{}]   {\specsymb_{\mathrm{#1}}(#2)}

\newcommand{\eps}{\varepsilon} 
\renewcommand{\phi}{\varphi}   
\renewcommand{\rho}{\varrho}   

\newcommand{\conj}[1]{\overline {#1}} 

\newcommand{\R}{\mathbb{R}} 
\newcommand{\C}{\mathbb{C}} 
\newcommand{\N}{\mathbb{N}} 

\newcommand{\1}{\mathbbm 1}                    

\newcommand{\e}{\mathrm e}  

\newcommand{\wt}{\widetilde}           

\newcommand{\leCS}{\stackrel{\mathrm{CS}}\le}
\newcommand{\leCY}{\stackrel{\mathrm{CY}}\le}

\newcommand{\geCY}{\stackrel{\mathrm{CY}}\ge}

\newcommand{\HS}{\mathscr H}           

\newcommand{\Contsymb} {\mathsf C}     
\newcommand{\Lsymb}    {\mathsf L}     
\newcommand{\lsymb}    {\ell}          

\newcommand{\Contspace}[1][{}]{\Contsymb^{#1}}     
\newcommand{\Lpspace}[1][p]    {\Lsymb_{#1}}     
\newcommand{\lpspace}[1][p]    {\lsymb_{#1}}     
\newcommand{\Lsqrspace}    {\Lpspace[2]}     
\newcommand{\lsqrspace}    {\lpspace[2]}          

\newcommand{\Cont}[2][{}]{\Contspace[#1]({#2})}

\newcommand{\Lsqr}[2][{}]{\Lsqrspace^{#1}({#2})} 
\newcommand{\lsqr}[2][{}]{\lsqrspace^{#1}({#2})}   

\newcommand{\Dir}{{\mathrm D}}              
 
\newcommand{\Err}{\mathrm O}

\newcommand{\quadtext}[1]{\quad\text{#1}\quad}
\newcommand{\qquadtext}[1]{\qquad\text{#1}\qquad}

\usepackage{nicefrac}

\newcommand{\Sierpinski}{Sierpi\'nski\xspace}

\newcommand{\pcf}{pcf\xspace}
\newcommand{\energy}{\mathcal E}
\newcommand{\eucl}{\mathrm{eucl}}
\newcommand{\conductance}{c}  

\newcommand{\deltaA}{\delta_{\mathrm a}}

\newcommand{\deltaC}{\delta_{\mathrm c}}

\newcommand{\resistMet}{\rho}
\newcommand{\lspace}[1]{\lsymb({#1})}   

\begin{document}

\title[Approximation of fractals by graphs] {Approximation of fractals
  by discrete graphs: norm resolvent and spectral convergence}

\author{Olaf Post}
\address{Fachbereich 4 -- Mathematik,
  Universit\"at Trier,
  54286 Trier, Germany}
\email{olaf.post@uni-trier.de}

\author{Jan Simmer}%
\address{Fachbereich 4 -- Mathematik,
  Universit\"at Trier,
  54286 Trier, Germany}
\email{simmer@uni-trier.de}

\date{\today, \thistime,  \emph{File:} \texttt{\jobname.tex}}

\begin{abstract}
  We show a norm convergence result for the Laplacian on a class of
  \pcf self-similar fractals with arbitrary Borel regular probability
  measure which can be approximated by a sequence of
  finite-dimensional weighted graph Laplacians.

  As a consequence other functions of the Laplacians (heat operator,
  spectral projections etc.) converge as well in operator norm.  One
  also deduces convergence of the spectrum and the eigenfunctions in
  energy norm.
\end{abstract}


\maketitle



%
\section{Introduction}
\label{sec:intro}
%
The aim of this article is to apply the abstract concept of
convergence of energy forms defined on different Hilbert spaces,
developed by the first author in~\cite{post:12,post:06} to the case of
certain self-similar fractals.  More precisely, this concept of
convergence of energy forms is based on \emph{quasi-unitary
  equivalence}, a (quantitative) generalisation of unitary equivalence
and norm resolvent convergence in the sense of \Rem{quasi-uni} (see
\Sec{norm.convergence} for further details).  We also give some
consequences of quasi-unitary equivalence in \CorS{main1}{main2}
below.  Details on self-similar fractals and precise definitions can
be found in~\cite{kigami:01,strichartz:06} going back
to~\cite{kigami:93a,kigami:93b}; see also \Sec{fractals} where we give
an informal presentation.

A \emph{self-similar set} $K \subset \R^d$ is generated by contractive
similarities $\map{F_1,\dots, F_N} {\R^d}{\R^d}$ such that
\begin{equation}
  \label{eq:similarities}
  K = \bigcup_{j=1,\dots,N} F_j(K).
\end{equation}

Our analysis works for a class of \pcf self-similar sets (see
\Def{pcf}) which we call here \emph{fractals approximable by finite
  weighted graphs} (see \Def{renorm}).  For such a fractal $K$ there
is a sequence $(G_m)_{m \in \N_0}$ of nested graphs $G_m=(V_m,E_m)$
(i.e., $V_m \subset V_{m+1} \subset K$) and conductances (i.e., edge
weights) $\conductance_{e,m}>0$ of the edges $e \in E_m$, such that
$V_*=\bigcup_m V_m$ is dense in $K$, (see e.g.\ \Fig{pentagasket5} for
the pentagasket with all five fixed points as boundary vertices)
\begin{figure}[h] 
  \centering
  \includegraphics[width=0.8\textwidth]{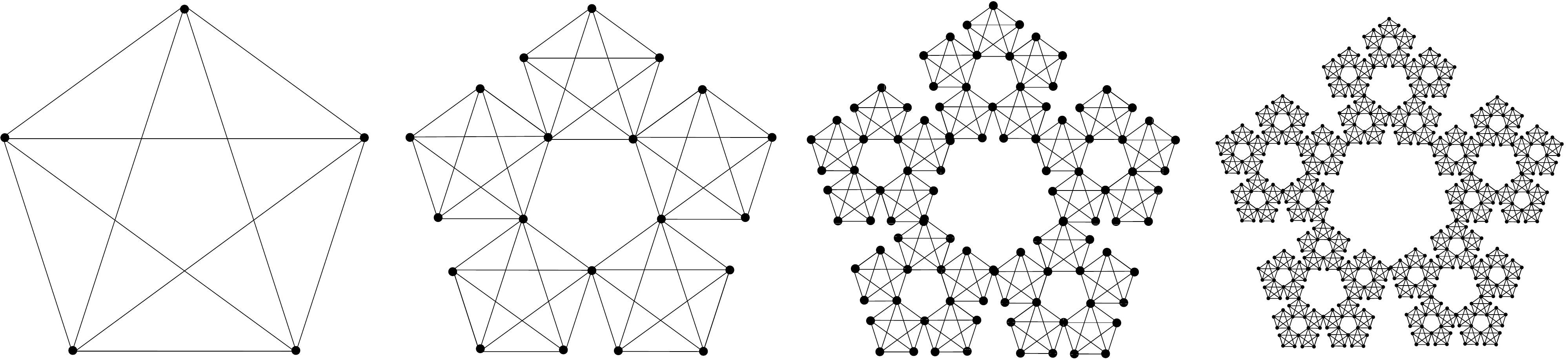}
\caption{The pentagasket approximation graphs $G_m$ starting with all
  five fixed points as boundary $V_0$ for the generations
  $m=0,\dots,3$ with $(\card {V_m})_m=(5, 5\cdot 5-5=20, 5\cdot
  20-5=95, 95 \cdot 5-5 =470, \dots)$.}
  \label{fig:pentagasket5}
\end{figure}
together with a \emph{compatible} and \emph{self-similar} sequence of
graph energies $(\energy_m)_m$ (see
\Defs{compatible}{forms.self-similar}) given by
\begin{equation}
  \label{eq:graph.energy}
  \energy_m(f)
  := \sum_{ e=\{x,y\} \in E_m} \conductance_{e,m} \abssqr{f(x)-f(y)}
\end{equation}
for $\map f {V_m} \C$.

The \emph{compatibility} roughly means that $\energy_m(\phi)$ agrees
with the energy of $\energy_{m+1}(h)$ where $\map h {V_{m+1}} \C$ is
the harmonic extension of $\map \phi {V_m} \C$.  The harmonic
extension $h$ has the property that it minimises the energy
$\energy_m(u)$ among all extensions $\map u {V_{m+1}} \C$ with $u
\restr{V_m}=\phi$, see~\eqref{eq:compatible}.

\emph{Self-similarity} means that $\energy_{m+1}$ can be constructed
from $\energy_m$ together with the renormalisation factors and contractive 
similarities defining the fractal $K$,
see~\eqref{eq:self-similar1}.

The existence of a compatible and self-similar sequence
$(\energy_m)_m$ (also referred to as the \emph{renormalisation
  problem}, see \Rem{renorm.prob}) is not guaranteed for a general
\pcf fractal, but the class of such fractals is sufficiently large to
give an interesting theory.  In particular, the \Sierpinski gasket,
its higher-dimensional analogs, the pentagasket and the hexagasket
(resp.\ any $n$-gasket if $n$ is not divisible by $4$) all belong to
the class of fractals approximable by finite weighted graphs.

The compatibility of $(\energy_m)_m$ guarantees that the limit
\begin{equation*}
  \wt \energy(u) 
  := \lim_{m \to \infty} \energy_m(u \restr {V_m})
\end{equation*}
exists for continuous functions $\map u K \C$ (and may be $\infty$).
The domain of $\wt \energy$ consists of those continuous functions for
which the limit is finite.  We call the non-negative closed quadratic
form $\wt \energy$ the \emph{energy form} of the fractal $K$ (and the
sequence of approximating graphs $(G_m)_m$).

We fix a Borel regular probability measure $\mu$ on $K$ having full
support (e.g.\ the homogeneous and self-similar measure associating
$\mu(F_j(K))=1/N$ to each self-similar image of $K$ or the Kusuoka
energy measure, see \Subsec{ex.measure}).  On each approximating graph
$G_m$, we define a discrete probability measure $\mu_m$, i.e., a
family $(\mu_m(x))_{x \in V_m}$ as the integral of the $m$-harmonic
function with boundary value $1$ at $x \in V_m$ and $0$ at $V_m
\setminus \{x\}$ (see~\eqref{eq:def.mu.m}; for a self-similar measure,
$\mu_m(x)$ decays exponentially in $m$).  We call the family
$(\mu_m)_m$ of probability measures on $G_m$ the \emph{approximating
  measures} corresponding to $(K,\mu)$.

\begin{remark*}
  We would like to stress that the definitions of energy forms and
  harmonic functions on $K$ and $G_m$ do not refer to a measure.  But
  once we are interested in the corresponding operators, eigenvalues or
  convergence results (as explained below), we need to fix a measure
  $\mu$ and choose for $(\mu_m)_m$ the approximating measures
  corresponding to $(K,\mu)$ as above.
\end{remark*}

Let $\wt \HS=\Lsqr{K,\mu}$ and $\HS_m=\lsqr{V_m,\mu_m}$ be the
corresponding Hilbert spaces, respectively.  Denote by $\wt \Delta$
resp.\ $\Delta_m$ the operators on the fractal resp.\ the discrete
graph associated with the energy forms $\wt \energy$ resp.\
$\energy_m$ in $\wt \HS$ resp.\ $\HS_m$.  The graph operator can
actually be represented as a matrix of increasing size, as the graphs
$G_m$ are finite.

For a fractal approximable by finite weighted graphs, the spectrum of
$\wt\Delta$ is purely discrete and hence consists of a sequence of
eigenvalues $(\lambda_k(\wt \Delta))_{k \in \N}$ in increasing order
and repeated according to the multiplicity.

\subsection{Main results}
Our first main result is the statement that the energy forms on
$(G_m,\mu_m)$ and $(K,\mu)$ are $\delta_m$-quasi-unitarily equivalent
with $\delta_m \to 0$ as $m \to \infty$.  The $\delta_m$-quasi-unitary
equivalence is formulated for quadratic forms acting in different
Hilbert spaces and it implies a variant of norm resolvent convergence
for the associated Laplacians.  Our method does not rely on the
monotonicity of the quadratic forms: It is well-known that an
increasing sequence of quadratic forms on a given Hilbert space $\HS$
converges to a limit form (see e.g.~\cite[Thm.~S.14]{reed-simon-1}),
and that the corresponding operators converge in strong resolvent
sense.  Here, we face a problem as the quadratic forms $\energy_m$ act
in $\HS_m$, and they are different for each $m \in \N$.  Moreover,
also the limit space $\wt \HS$ is different, and it is a priori not
clear how to relate these spaces.  We comment on other approaches like
Mosco convergence of forms in \Subsec{prev.works}.

The proof of the following first main result will be given in
\Subsec{approx.fractals.graphs}:
\begin{theorem}
  \label{thm:main}
  Let $(K,\mu)$ be a fractal approximable by finite weighted graphs
  $(G_m,\mu_m)$ with Borel regular probability measure $\mu$ of full
  support $K$.  Then the fractal energy form $\energy_K$ and the
  approximating graph energy forms $\energy_m$ are
  $\delta_m$-quasi-unitarily equivalent on $\wt\HS=\Lsqr{K,\mu}$ and
  $\HS_m=\lsqr{G_m,\mu_m}$, where $\delta_m \to 0$ as $m \to \infty$
  is given precisely in~\eqref{eq:main.delta}.  In particular,
  $\delta_m$ decays exponentially fast.
\end{theorem}

The definition of quasi-unitary equivalence needs an
\emph{identification operator}
\begin{equation*}
  \map {J=J_m} {\HS_m= \lsqr{V_m,\mu_m}} {\wt \HS=\Lsqr{K,\mu}}
\end{equation*}
such that the resolvent weighted operator norms fulfil
\begin{subequations}
  \begin{equation}
    \label{eq:que.op}
    \norm{(\id_{\HS_m}-J^*J)(\Delta_m+1)^{-1/2}} \le \delta_m
    \qquadtext{and}
    \norm{(\id_{\wt \HS} - J J^*)(\wt \Delta+1)^{-1/2}} \le \delta_m,
  \end{equation}
  hence tend to $0$.  We usually suppress the dependence of $J$ on $m$
  in the notation; in particular, the second operator norm also
  depends on $m$.  The definition of quasi-unitary equivalence
  (\Def{quasi-uni}) is more elaborated, as it uses also identification
  operators on the form domain level and is expressed entirely in
  terms of the energy forms.  On the operator level,
  $\delta_m$-quasi-unitary equivalence implies that
  \begin{equation}
    \label{eq:que.op'}
    \norm{(\wt \Delta+1)^{-1}J-J(\Delta_m+1)^{-1}} 
    \le \delta_m,
  \end{equation}
\end{subequations}
see~\cite[Prp.~4.4.15]{post:12}.  Once quasi-unitary equivalence and
hence~\eqref{eq:que.op}--\eqref{eq:que.op'} are fulfilled, we can
mimic the proofs of results following from norm resolvent convergence
(which actually is the convergence of the norm~\eqref{eq:que.op'} to
$0$ for $\wt \HS=\HS_m$ and $J=\id$).  These consequences are treated
systematically in~\cite[Ch.~4]{post:12}; we just cite two of them here
(see~\cite[Thm.~4.2.14 and Thm.~4.3.5]{post:12}):

The first consequence states in our setting that a function of the
operator on the fractal is close to the corresponding operator on the
graph sandwiched by $J^*$ and $J$ and vice versa:
\begin{corollary}
  \label{cor:main1}
  Let $\map \eta {[0,\infty)} \C$ be a function continuous in a
  neighbourhood $U$ of $\spec {\wt \Delta}$ such that $\lim_{\lambda
    \to \infty} (\lambda+1)^{1/2} \eta(\lambda)$ exists.  Then there
  is a function $\delta \to \eps(\delta)$ such that $\eps(\delta)\to
  0$ as $\delta\to 0$, depending only on $\eta$ and $U$, and
  \begin{equation*}
    \norm{\eta(\wt \Delta) - J \eta(\Delta_m) J^*}
    \le \eps(\delta_m)
    \quadtext{and}
    \norm{\eta(\Delta_m) - J^* \eta(\wt \Delta) J} 
    \le \eps(\delta_m).
  \end{equation*}
\end{corollary}
For a certain class of functions $\eta$, we can actually show that
$\eps$ is Lipschitz, i.e., that $\eps(\delta)=C_{\eta,U} \delta$.  For
example, $\eta(\lambda)=\e^{-t\lambda}$ or $\eta=\1_I$ with $\bd I
\cap \spec {\wt\Delta} = \emptyset$ belong to this class.  In particular,
\Cor{main1} then states the norm convergence of the approximating heat
operators on the discrete graphs to the heat operators on the fractal.
If $\eta=\1_I$ then we have the following statement about the spectral
projections:
\begin{corollary}
  \label{cor:main1a}
  Let $I$ be an interval in $\R$ such that $\bd I \cap \spec {\wt
    \Delta} = \emptyset$.  Then the spectral projections converge in
  operator norm, i.e.,
  \begin{equation*}
    \norm{\1_I(\wt \Delta) - J \1_I(\Delta_m) J^*} \le C_I \delta_m
  \end{equation*}
  where the constant $C_I$ depends only on $I$ and the distance of
  $\bd I$ to $\spec{\wt \Delta}$.
\end{corollary}

As in the case of usual operator norm convergence, the operator norm
convergence of spectral projections as in \Cor{main1a} also implies
the convergence of eigenvalues (we give a direct proof in
\Thm{ew.direct} using the min-max characterisation of eigenvalues in
\Prp{q-u-e.ew}, under slightly stronger assumptions).  Such an
eigenvalue convergence result is known as folklore, at least for
fractals with self-similar measure, and where the spectral decimation
method is available (see e.g.~\cite{fukushima-shima:92}), but seems to
be new in other cases or if the measure is not self-similar.
\begin{corollary}
  \label{cor:main2}
  Let $\lambda_k(\Delta_m)$ resp.\ $\lambda_k(\wt \Delta)$ be the
  $k$-th eigenvalue of $\Delta_m$ resp.\ $\wt \Delta$ (in increasing
  order, repeated according to multiplicity).  Then
  \begin{equation*}
    \bigabs{\lambda_k(\Delta_m)-\lambda_k(\wt \Delta)}
    \le C_k \delta_m
  \end{equation*}
  for all $m \in \N$ such that $\dim \HS_m \ge k$, where $C_k$ depends
  only on $\lambda_k(\wt \Delta)$ and where $\delta_m \to 0$ is the
  same as in \Thm{main}, i.e., exponentially decaying.
\end{corollary}
In the case of purely discrete spectrum or the case of isolated
eigenvalues, we can approximate an eigenfunction also in \emph{energy
  norm} by a sequence of finite dimensional eigenvectors: namely, for
an isolated eigenvalue $\wt \lambda$ with normalised eigenfunction
$\wt \Phi$, there is a sequence $(\Phi_m)_m$ of normalised functions
consisting of a linear combination of eigenfunctions with eigenvalues
close to $\wt \lambda$ such that
\begin{equation}
  \label{eq:ef.conv}
  \norm[\dom \wt \energy]{J \Phi_m - \wt \Phi} \le C_{\wt \lambda} \delta_m
\end{equation}
where $C_{\wt \lambda}$ depends only on $\wt \lambda$.  Here,
$\normsqr[\dom \wt \energy] u = \normsqr[\Lsqr{K,\mu}]u +\wt
\energy(u)$ denotes the energy norm.  It follows that the convergence
also holds with respect to the original Hilbert space norm of $\wt
\HS$.  We assume here tacitly that the range of $J$ lies in $\dom \wt
\energy$; the more general case is treated in \Prp{conv.ef}.

\begin{remark}
  \label{rem:fractafold}
  We would like to stress that a result similar to \Thm{main} holds
  also for non-compact spaces such as non-compact \emph{fractafolds}.
  A \emph{fractafold} is a space which is locally homeomorphic to a
  given \pcf self-similar fractal $K$ (see e.g.~\cite{strichartz:03}).
  A non-compact example is given by $X=\bigcup_{m \in \N} F^{-m}(K)$
  for some \pcf self-similar fractal $K$ with IFS $F$, where
  $F^{-m}(K)=\bigcup_{w \in W_m} F_w^{-1}(K)$ for $m \ge 0$ (with the
  convention that $F_w^{-1}$ is the inverse of $F_w$).  We also let
  $G_m$ be the associated $m$-level graph approximation (here, the
  vertices of $G_m$ are included in the vertices of $G_{m+1}$, and
  each graph $G_m$ is infinite).  In particular, it follows from the
  abstract theory of~\cite[Sec.~4.3]{post:12} that the entire spectrum
  and the essential spectrum converge.  Moreover, we conclude that
  isolated eigenvalues and their eigenfunctions converge in the sense
  of~\eqref{eq:ef.conv} (see again \Prp{conv.ef} for the precise
  statement).
\end{remark}

\subsection{Previous works and further developments}
\label{ssec:prev.works}
In~\cite{fukushima-shima:92} the method of \emph{spectral decimation}
is described for the first time.  Roughly speaking, one can calculate
with this method the eigenvalues of generation $m+1$ from generation
$m$ by the preimage of some rational function.  From the spectral
decimation method, one can also conclude the convergence of
eigenvalues, cf.~\cite{fukushima-shima:92}
or~\cite[Sec.~3.1]{shima:96}.  This method is restricted to certain
fractals, and works only for the self-similar measure.  Our method
works for \emph{any} Borel regular probability measure of full
support.

Mosco~\cite{mosco:94} developed a notion of convergence of energy
forms, nowadays called ``Mosco convergence'' which is equivalent to
strong resolvent convergence.  Kuwae and Shioya~\cite[Sec.~2,
esp.~Sec.~2.5]{kuwae-shioya:03} extended the notion of Mosco
convergence to the case of varying Hilbert spaces.  Our concept
corresponds to a generalisation of \emph{norm} resolvent convergence,
see \Rem{homo} for the relation with homogenisation theory.
In~\cite{hinz-teplyaev:16}, there is a quite general approach how
(separable) Dirichlet forms can be approximated by Dirichlet forms on
finite spaces such that the corresponding forms converge in the sense
of Mosco.  We will relate Mosco convergence and our generalised norm
resolvent convergence in a subsequent publication.

In~\cite{iprrs:10} the resolvent for the Neumann resp.\ Dirichlet
Laplacian is calculated for \pcf self-similar fractals $K$.  The
Neumann Laplacian is the operator associated with $\energy_K$, while
the Dirichlet Laplacian is the operator associated with
$\energy_K^\Dir := \energy_K \restr {\set{u \in \dom \energy_K}{u
    \restr {V_0}=0}}$.

Our approach only needs the existence of a sequence of graphs and
energy forms converging in a suitable sense to the limit space and
limit energy form.  In particular, the self-similar structure of our
fractal is not essential for the proof of \Thm{main}, and one can
extend our results to certain finitely ramified fractals as introduced
in~\cite{teplyaev:08} (see also \Rem{fin.ram.frac}).

For numerical calculations of spectra, an adaption of the finite
element method (FEM) has been used also for fractals, e.g.\
in~\cite{grs:01,asst:01}.  We comment on this approach in \Rem{fem}.
Moreover, some of the conditions we have to check for the
quasi-unitary equivalence appeared already in the literature
in the context of the FEM.  For example, let $P_mu$ be the piecewise
harmonic interpolation spline (using the terminology of Strichartz et
al~\cite{strichartz-usher:00,grs:01}), i.e., the function with values
$u(x)$ for $x \in V_m$ and being $m$-harmonic on $K$.  Then
$P_mu=JJ^{\prime 1}u$ in our terminology of
\Subsec{approx.fractals.graphs}.  Moreover, the estimate
\begin{equation*}
  \normsqr[\Lsqr{K,\mu}]{u- P_m u}
  \le \frac{r_0}N \energy_K(u)
\end{equation*}
for $u \in \dom \energy_K$ for symmetric fractals (with common energy
renormalisation factor $r_0 \in (0,1)$) appearing in the proof of
\Thm{main} in \Subsec{approx.fractals.graphs} is already shown
in~~\cite[Thm.~3.4]{grs:01}.

We will apply the concept of quasi-unitary equivalence also to the
case of ``graph-like'' spaces such as metric graphs and graph-like
manifolds in a second publication~\cite{post-simmer:pre18}. We will
show that under suitable assumptions, a fractal energy form can be
approximated by a suitably renormalised energy form on a family of
metric graphs or graph-like spaces.  This will complement some
numerical results of ``outer'' approximations resp.\ approximations of
fractals by open domains with their natural Neumann Laplacians as
discussed e.g.~in \cite[pp~50--51,
Fig.14-15]{bsu:08} 
and~\cite{bhs:09}. 

We are confident that our method also works for approximations of
\emph{magnetic} Laplacians on fractals by discrete magnetic Laplacians
as described
in~\cite{hinz-teplyaev:13,hinz-rogers:16,hkmrs:17,bchlr:pre17}, hence
allowing to show quasi-unitary equivalence for more general fractals
(and not only those where the spectral decimation method is valid).

\subsection{Structure of the article}
In \Sec{norm.convergence}, we give a brief review of the abstract
convergence result for energy forms in different Hilbert spaces.
\Sec{fractals} contains a brief introduction to \pcf fractals in the
way we need it.  \Sec{approx.fractals.graphs} --- the core of the
paper --- contains the proofs of our main results.  In \Sec{examples}
we present some examples.

\subsection{Acknowledgements}
We would like to thank Michael Hinz for helpful comments on the
manuscript and literature.  We would also like to thank the anonymous
referee for carefully reading the manuscript and helpful comments.

%
\section{An abstract norm resolvent convergence result}
\label{sec:norm.convergence}
%

\subsection{Quasi-unitary equivalence}
In this section, we define a sort of ``distance'' between two
operators $\Delta$ and $\wt \Delta$ defined via energy forms $\energy$
and $\wt \energy$, which act in different Hilbert spaces $\HS$ and
$\wt \HS$, respectively.  The general theory is developed in great
detail in~\cite[Ch.~4]{post:12}.  The concept of quasi-unitary
equivalence assures a \emph{generalised} norm resolvent convergence
for operators $\Delta=\Delta_m$ converging to $\wt \Delta$ as $m \to
\infty$.  Each operator $\Delta_m$ acts in a Hilbert space $\HS_m$ for
$m \in \N$; and the Hilbert spaces are allowed to depend on $m$.

In our application, the Hilbert spaces $\HS=\HS_m=\lsqr{V_m,\mu_m}$ are
functions on the vertices $V_m$ of a weighted graph $(G_m,\mu_m)$.
Moreover, the ``limit'' metric measure space is $(X,\mu)$ with Hilbert
space $\wt \HS=\Lsqr{X,\mu}$.\footnote{It is a question of
  interpretation, which object is simpler and which one wants to
  approximate.  In the first application of the concept of
  quasi-unitary equivalence, the space $\HS$ was the limit space, the
  $\Lsqrspace$-space over a metric graph, while $\wt \HS$ was based on
  a family of shrinking (``graph-like'') manifolds.}

We start now with the general concept:

Let $\HS$ and $\wt \HS$ be two separable Hilbert spaces.  We say that
$\energy$ is an \emph{energy form in $\HS$} if $\energy$ is a closed,
non-negative quadratic form in $\HS$, i.e., if $\energy(f)\coloneqq
\energy(f,f)$ for some sesquilinear form $\map {\energy}{\HS^1 \times
  \HS^1} \C$, denoted by the same symbol, if $\energy(f)\ge 0$ and if
$\HS^1:=\dom \energy$, endowed with the norm defined by
\begin{equation}
  \label{eq:qf.norm}
  \normsqr[1] f
  \coloneqq \normsqr[\HS] f + \energy(f),
\end{equation}
is itself a Hilbert space and dense (as a set) in $\HS$.  We call the
corresponding non-negative, self-adjoint operator $\Delta$ (see
e.g.~\cite[Sec.~VI.2]{kato:66}) the \emph{energy operator} associated
with $\energy$.  Similarly, let $\wt \energy$ be an energy form in
$\wt \HS$ with energy operator $\wt \Delta$.

Associated with an energy operator $\Delta$, we introduce a natural
\emph{scale of Hilbert spaces} $\HS^k$ defined via the \emph{abstract
  Sobolev norms}
\begin{equation}
  \label{eq:def.abstr.sob.norm}
  \norm[k] f 
  \coloneqq \norm{(\Delta+1)^{k/2}f}.
\end{equation}
Then $\HS^k=\dom \Delta^{k/2}$ if $k \ge 0$ and $\HS^k$ is the
completion of $\HS=\HS^0$ with respect to the norm $\norm[k] \cdot$
for $k<0$.  Obviously, the scale of Hilbert spaces for $k=1$ and its
associated norm agrees with $\HS^1$ and $\norm[1]\cdot$ defined above
(see~\cite[Sec.~3.2]{post:12} for details).  Similarly, we denote by
$\wt \HS^k$ the scale of Hilbert spaces associated with $\wt \Delta$.

We now need pairs of so-called \emph{identification operators} acting
on the Hilbert spaces and later also pairs of identification operators
acting on the form domains.
\begin{definition}
  \label{def:quasi-uni}
  \begin{subequations}
    \label{eq:quasi-uni}
    Let $\delta \ge 0$, and let $\map J \HS {\wt \HS}$ and $\map {J'}
    {\wt \HS}\HS$, resp.\ $\map {J^1} {\HS^1} {\wt \HS^1}$ and $\map
    {J^{\prime1}} {\wt \HS^1}{\HS^1}$ be bounded linear operators on
    the Hilbert spaces and energy form domains.
    \begin{enumerate}
    \item We say that $J$ is \emph{$\delta$-quasi-unitary} with
      \emph{$\delta$-quasi-adjoint} $J'$ if
      \begin{align}
        \label{eq:quasi-uni.a}
        &\norm{Jf}\le (1+\delta) \norm f,
        &\bigabs{\iprod {J f} u - \iprod f {J' u}}
        &\le \delta \norm f \norm u
        &(f \in \HS, u \in \wt \HS),\\
        \label{eq:quasi-uni.b}
        &\norm{f - J'Jf}
        \le \delta \norm[1] f,
        &\norm{u - JJ'u}
        &\le \delta \norm[1] u
        &(f \in \HS^1, u \in \wt \HS^1).
      \end{align}
      
    \item We say that $J^1$ and $J^{\prime1}$ are
      \emph{$\delta$-compatible} with the identification operators $J$
      and $J'$ if
      \begin{equation}
        \label{eq:quasi-uni.c}
        \norm{J^1f - Jf}\le \delta \norm[1]f, \quad
        \norm{J^{\prime1}u - J'u} \le \delta \norm[1] u
        \qquad (f \in \HS^1, u \in \wt \HS^1).
      \end{equation}
      
    \item We say that the energy forms $\energy$ and $\wt \energy$ are
      \emph{$\delta$-close} if
      \begin{equation}
        \label{eq:quasi-uni.d}
        \bigabs{\wt \energy(J^1f, u) - \energy(f, J^{\prime1}u)} 
        \le \delta \norm[1] f \norm[1] u
        \qquad (f \in \HS^1, u \in \wt \HS^1).
      \end{equation}
      
    \item We say that $\energy$ and $\wt \energy$ are
      \emph{$\delta$-quasi-unitarily equivalent} (on $\HS$ and $\wt
      \HS$), if~\eqref{eq:quasi-uni.a}--\eqref{eq:quasi-uni.d} are
      fulfilled, i.e.,
      \begin{itemize}
      \item if there exists identification operators $J$ and $J'$ such
        that $J$ is $\delta$-quasi-unitary with $\delta$-adjoint $J'$
        (i.e., \eqref{eq:quasi-uni.a}--\eqref{eq:quasi-uni.b} hold);
      \item if there exists identification operators $J^1$ and
        $J^{\prime1}$ which are $\delta$-compatible with $J$ and $J'$
        (i.e., \eqref{eq:quasi-uni.c} holds);
      \item and if $\energy$ and $\wt \energy$ are $\delta$-close (i.e.,
        \eqref{eq:quasi-uni.d} holds).
      \end{itemize}
    \end{enumerate}
  \end{subequations}
\end{definition}

At first sight the previous definition looks a bit technical, but we
will see in \Sec{approx.fractals.graphs} that it fits perfectly for
our approximation of energy forms on fractals by energy forms on
graphs and is easy to check.  Note that~\eqref{eq:quasi-uni.b} is
equivalent with~\eqref{eq:que.op} in operator norm notation.

\begin{remark}
  \label{rem:quasi-uni}
  Let us state some trivial consequences, explaining the notation in
  two extreme cases:
  \begin{enumerate}
  \item \emph{``$\delta$-quasi-unitary equivalence'' is a quantitative
    generalisation of ``unitary equivalence'':}

  Note that if $\delta=0$, $J$ is $0$-quasi-unitary if and only if $J$
  is unitary with $J^*=J'$.  Moreover, $\energy$ and $\wt \energy$ are
  $0$-quasi-unitarily equivalent if and only if $\Delta$ and $\wt
  \Delta$ are unitarily equivalent (in the sense that $JR=\wt RJ$).

\item \emph{``$\delta_m$-quasi-unitary equivalence'' with $\delta_m
    \to 0$ is a generalisation of ``norm resolvent convergence'':} If
  $\HS=\wt \HS$, $J=J'=\id_\HS$, then the first two
  conditions~\eqref{eq:quasi-uni.a}--\eqref{eq:quasi-uni.b} are
  trivially fulfilled with $\delta=0$.  Moreover, if $\Delta=\Delta_m$
  and $\delta=\delta_m \to 0$ as $m \to \infty$, then $\wt \energy$
  and $\energy_m$ are $\delta_m$-quasi-unitarily equivalent if and
  only if $\norm{R_m^{-1/2}(R_m - \wt R)\wt R^{-1/2}} \to 0$ as $m \to
  \infty$, and hence it follows that $\norm{R_m - \wt R} \to 0$, i.e.,
  $\Delta_m$ \emph{converges} to $\wt \Delta$ \emph{in norm resolvent
    sense} as $m \to \infty$.
  \end{enumerate}
\end{remark}

\begin{remark}
  \label{rem:homo}
  If $\energy_m$ and $\wt \energy$ are $\delta_m$-quasi-unitarily
  equivalent, then the corresponding operators also converge in the
  sense of Kuwae and Shioya~\cite[Sec.~2.5]{kuwae-shioya:03}.  Our
  concept is a generalisation of \emph{norm} resolvent convergence
  while Kuwae and Shioya provide a generalisation of \emph{strong}
  resolvent convergence, i.e., convergence of the resolvents in the
  pointwise operator topology to the case of varying Hilbert spaces.
  Their convergence is equivalent to a version of Mosco convergence of
  forms in varying spaces, see~\cite[Thm.~2.4]{kuwae-shioya:03} .  We
  discuss such questions in a subsequent paper.

  Similarly, in homogenisation problems, usually strong resolvent
  convergence is shown, and hence a typical assumption is compactness
  of the resolvent (i.e., compactness of the domain).
  In~\cite{khrabustovskyi-post:18} the first author and Khrabustovskyi
  show that in the situation of the Dirichlet Laplacian on a domain
  with many small periodic obstacles removed one has
  $\delta_\eps$-quasi-unitary equivalence with a homogenised operator,
  where $\eps$ is the order of periodicity.  Apart from~\cite{dcr:18}
  this is the first time (to the best of our knowledge) where
  \emph{norm} resolvent convergence is shown directly for
  homogenisation problems.

  Moreover, although the concept of quasi-unitary equivalence is
  stronger than e.g.~Mosco convergence, it seems to be not very
  difficult to show the conditions in \Def{quasi-uni} in applications
  such at graphs approximating self-similar fractals or in the above
  homogenisation setting of~\cite{khrabustovskyi-post:18}.
\end{remark}

Note that the conditions in Definition~\ref{def:quasi-uni} are not
written in stone.  It will be convenient in our case to use the
following modification:
\begin{lemma}
  \label{lem:quasi-uni.b}
  Assume that~\eqref{eq:quasi-uni.a} is fulfilled with
  $\deltaA>0$ and~\eqref{eq:quasi-uni.c} with
  $\deltaC>0$. If
  \begin{equation}
    \label{eq:quasi-uni.b''}
    \norm{u - JJ^{\prime1}u}
    \le \delta' \norm[1] u \qquad (u \in \wt \HS^1)
  \end{equation}
  holds, then the second inequality in~\eqref{eq:quasi-uni.b} is
  fulfilled with $\delta=\delta'+(1+\deltaA) \deltaC$.
  
  In particular, if all conditions~\eqref{eq:quasi-uni} are fulfilled
  for some $\delta>0$, except for the second one
  in~\eqref{eq:quasi-uni.b} which is replaced
  by~\eqref{eq:quasi-uni.b''}, then $\energy$ and $\wt \energy$ are
  $\wt \delta$-quasi-unitarily equivalent with $\wt \delta=\delta' +
  (1+\delta)\delta$.
\end{lemma}
\begin{proof}
  We have
  \begin{align*}
    \norm{u - JJ'u}
    &\le \norm{u - J J^{\prime1} u}
    + \norm{J(J^{\prime1}-J')u}\\
    &\le \norm{u - JJ^{\prime1}u}
    + \norm J \norm{(J^{\prime1}-J')u}
    \le \bigl(\delta'+(1+\deltaA)\deltaC\bigr) \norm[1] u
  \end{align*}
  and if $\delta',\deltaA,\deltaC \le \delta$,
  then the error estimate is greater or equal to $2\delta+\delta^2$,
  as claimed.
\end{proof}

\subsection{Consequences of quasi-unitary equivalence}

Let us mention some consequences of the above-mentioned quasi-unitary
equivalence we need in the proof of \Prp{conv.ef}.

We first provide the notation $\norm[-1\to1]A=\norm{(\wt
  \Delta+1)^{1/2}A(\Delta+1)^{1/2}}$ for $\map A{\HS^{-1}}{\wt\HS^1}$.
Moreover, $\map{(J^{\prime 1})^*}{\HS^{-1}}{\wt \HS^{-1}}$ denotes the
dual map of $\map{J^{\prime 1}}{\wt \HS^1}{\HS^1}$ with respect to the
dual pairing $\HS^1 \times \HS^{-1}$ induced by the inner product on
$\HS$ and similarly on $\wt \HS$.
\begin{proposition}[{see~\cite[Ch.~4]{post:12}}]
  \label{prp:que-conseq}
  Let $\map \eta {[0,\infty)} \C$ be a function continuous in a
  neighbourhood $U$ of $\spec {\wt \Delta}$ such that $\lim_{\lambda
    \to \infty} (\lambda+1)^{1/2} \eta(\lambda)$ exists.  Then there
  is a function $\delta \mapsto \eps(\delta)$ such that
  $\eps(\delta)\to 0$ if $\delta \to 0$ and
  \begin{subequations}
    \begin{align}
      \label{eq:que-j'}
      \norm {J' \eta(\wt \Delta) - \eta(\Delta) J'}
      &\le \eps(\delta),\\
      \label{eq:que-jnorm}
      \bigabs{\norm{J'u}-\norm u}
      &\le 3\delta \norm[1] u \\
      \label{eq:que-j''}
      \norm[-1\to 1]{J^1\eta(\Delta) - \eta(\wt \Delta) (J^{\prime 1})^*}
      &\le \eps(\delta)
    \end{align}
  \end{subequations}
  for any pair of $\delta$-quasi-unitarily equivalent energy forms
  $\energy$ and $\wt\energy$ with associated operators $\Delta$ and
  $\wt \Delta$, where $\eps=\eps_\eta$ depends only on $\eta$ and $U$.
  If $\eta(\lambda)=\e^{-t \lambda}$ (for $t >0$) or $\eta=\1_D$ (if
  $\bd D \cap \spec {\wt \Delta}=\emptyset$ for an open set $D \subset
  \C$ with smooth boundary) we can choose $\eps(\delta)=C \delta$ for
  some constant $C>0$.
\end{proposition}

An inequality similar to~\eqref{eq:que-jnorm} is proven
in~\eqref{eq:que-jnorm'} in a more specific situation.

\begin{proposition}
  \label{prp:conv.ef}
  Let $\energy$ and $\wt {\energy}$ be two $\delta$-quasi-unitarily
  equivalent energy forms with associated operators $\Delta$ and $\wt
  \Delta$.  Assume that $\wt \Phi$ is a normalised eigenvector of $\wt
  \Delta$, such that its eigenvalue $\wt \lambda$ is discrete in
  $\spec {\wt \Delta}$, i.e., there is an open disc $D$ in $\C$ such
  that $\spec {\wt \Delta} \cap D=\{\wt \lambda\}$.  Then there exists
  a normalised eigenvector $\Phi$ of $\Delta$ with $\Phi \in \ran
  \1_D(\Delta)$ and a universal constant $C$ depending only on $\wt
  \lambda$ (and the radius of $D$) such that
  \begin{equation*}
    \norm[1]{J^1 \Phi - \wt \Phi} \le C\delta
  \end{equation*}
  for $\delta>0$ small enough.
\end{proposition}
Note that the eigenvalue $\wt \lambda$ does not necessarily need to
have finite multiplicity.
\begin{proof}
  Set $P:=\1_D(\Delta)$ and $\wt P:=\1_D(\wt\Delta)$.  Note that $\ran
  P$ may consist of the linear combination of several eigenvectors if
  $\wt \lambda$ is not a simple eigenvalue.  We have
  \begin{equation*}
    \norm{PJ'\wt \Phi}
    \ge \norm{J'\wt P\wt\Phi} - \norm{(PJ'-J'\wt P)\wt\Phi}
    \ge \norm{J'\wt\Phi} - C_\eta \delta \norm {\wt \Phi}
    \ge 1-\underbrace{\bigl(3(\wt \lambda+1)^{1/2}+C_\eta\bigr)}_{=:C_1}\delta
  \end{equation*}
  by~\eqref{eq:que-j'} with $\eta=\1_D$ and~\eqref{eq:que-jnorm} and similarly,
  \begin{equation*}
    \norm{PJ'\wt \Phi}
    \le \norm{J'\wt P\wt\Phi} + \norm{(PJ'-J'\wt P)\wt\Phi}
    \le \norm{J'\wt\Phi} + C_\eta \delta \norm {\wt \Phi}
    \le 1+C_1 \delta,
  \end{equation*}
  and therefore
  \begin{equation}
    \label{eq:conv.ef}
    \bigabs{\norm{PJ'\wt \Phi}-1} \le C_1 \delta.
  \end{equation}
  In
  particular, $\norm{PJ'\wt\Phi} >0$ for $\delta$ small enough.
  Let
  \begin{equation*}
    \Phi := \frac 1 {\norm {PJ'\wt\Phi}} P J' \wt \Phi,
  \end{equation*}
  then $\Phi \in \HS^1$ and for $\delta<1/C_1$ we have
  \begin{align*}
    \norm[1]{J^1 \Phi - \wt \Phi}
    &= 
    \Bignorm[1]{\frac 1 {\norm {PJ'\wt\Phi}} J^1P J' \wt \Phi - \wt \Phi}\\
    &\le\frac 1 {\norm {PJ'\wt\Phi}} 
    \Bigl( 
      \bignorm[1]{(J^1P - \wt P (J^{\prime 1})^*)J'\wt\Phi}
      + \bignorm[1]{\wt P((J^{\prime 1})^*J'\wt \Phi - \wt \Phi)}
      \\
    &\hspace*{0.515\textwidth}+ \bigabs{1-\norm{PJ'\wt \Phi}} \norm[1]{\wt \Phi}
    \Bigr)\\
    &\le\frac 1 {1-C_1\delta} 
    \Bigl( 
      C_\eta \delta \norm[-1]{\wt \Phi}
      + (\wt\lambda+1) \bignorm[0\to-1]{(J^{\prime 1})^*J'-\id} \norm[1]{\wt \Phi}
      + C_1\delta \norm[1]{\wt \Phi}
    \Bigr)\\
    &\le\frac 1 {1-C_1\delta} 
    \bigl(C_\eta + (\wt\lambda+1)^{3/2} C'' + C_1 (\wt \lambda+1)^{1/2}
    \bigr)\delta
  \end{align*}
  using~\eqref{eq:que-j''}, $\norm[-1 \to 1]{\wt P}=1+\wt \lambda$
  and~\eqref{eq:conv.ef} for the second inequality, and
  \begin{align*}
    \bignorm[0\to-1]{(J^{\prime 1})^*J'-\id}
    =\bignorm[1\to 0]{(J')^*J^{\prime 1}-\id}
    &\le \norm{(J')^*-J}\norm[1 \to 0]{J^{\prime 1}}
      + \norm[1\to 0]{J J^{\prime 1}-\id}\\
      &\le \delta (1+3\delta) + \delta' 
      \le (1+3/C_1 + C') \delta
      =: C'' \delta
  \end{align*}
  using also
  \begin{equation*}
    \norm[1 \to 0]{J^{\prime 1}} 
    \le \norm[1 \to 0]{J^{\prime 1}-J'} + \norm[0 \to 0]{J'-J^*} 
    + \norm[0 \to 0]{J^*}
    \le 1+3\delta
    \le 1+3/C_1
  \end{equation*}
  as $\norm[1 \to 0] A \le \norm[0 \to 0] A$ and $\delta<1/C_1$, and
  using $\delta'\le C'\delta$ for some constant $C'>0$; the latter
  estimate can be seen similarly as in \Lem{quasi-uni.b}.
\end{proof}

\subsection{Some eigenvalue estimates}
For problems with purely discrete spectrum, we also cite some results
dealing directly with eigenvalue estimates, using the min-max
principle (see~\cite[Prp.~4.4.18]{post:12} for the most general
version): Here, $\lambda_k$ denotes the $k$-th eigenvalue of the
operator associated with $\energy$, in increasing order and repeated
according to their multiplicity.
\begin{proposition}
  \label{prp:ew.comp}
  Assume that we have a first order identification operator
  \begin{equation*}
    \map {J^1} {\HS^1} {\wt \HS^1}
  \end{equation*}
  such that there exist $\delta_0 \ge 0$ and $\delta_1 \ge 0$ with
  \begin{equation}
    \label{eq:ew.comp}
    \normsqr{J^1 f} 
    \ge \normsqr f 
    - \delta_0\normsqr f - \delta_1 \energy(f)
    \quadtext{and}
    \wt {\energy} (J^1 f)
    \le \energy(f) 
  \end{equation}
  for all $f \in \HS^1$. If $\delta_0+\delta_1 \lambda_k < 1$ then
  \begin{equation}
    \label{eq:ew.est}
    \wt \lambda_k 
    \le \frac 1{1-\delta_0-\delta_1 \lambda_k} 
    \cdot \lambda_k
    \quadtext{or}
    \frac {1-\delta_0}{1+\delta_1 \wt \lambda_k} \cdot \wt \lambda_k
    \le \lambda_k
  \end{equation}
  for all $k \in \N$.
\end{proposition}

If we have shown quasi-unitary equivalence and one additional
assumption, then we can conclude from the last proposition (again
slightly adopted to our application in \Sec{approx.fractals.graphs} in
order to get an optimal estimate):
\begin{proposition}
  \label{prp:q-u-e.ew}
  Assume that $\energy$ and $\wt \energy$ are $\delta$-quasi-unitarily
  equivalent (with~\eqref{eq:quasi-uni.a} and~\eqref{eq:quasi-uni.d}
  fulfilled with $\delta=0$ and~\eqref{eq:quasi-uni.b}
  and~\eqref{eq:quasi-uni.c} with $\norm[1] f$ resp.\ $\norm[1] u$
  replaced by $\energy(f)^{1/2}$ resp.\ $\wt \energy(u)^{1/2}$) and
  that
  \begin{subequations}
    \begin{equation}
      \label{eq:q-u-e.ew}
      \energy(f,J^{\prime 1} J^1 f-f)\le 0
    \end{equation}
    for all $f \in \HS^1$, then~\eqref{eq:ew.est} is fulfilled with
    $\delta_0=\delta_1=\delta$.  Moreover, if
    \begin{equation}
      \label{eq:q-u-e.ew'}
      \wt \energy(u,J^1J^{\prime 1} u-u)\le 0
    \end{equation}
    for all $u \in \wt \HS^1$, then
    \begin{equation}
      \label{eq:q-u-e.ew.result}
      \lambda_k 
      \le \frac 1{1-\delta(1+\wt \lambda_k)} 
      \cdot \wt \lambda_k
      \quadtext{or}
      \frac {1-\delta}{1+\delta \lambda_k} \cdot \lambda_k
      \le \wt \lambda_k.
    \end{equation}
  \end{subequations}
\end{proposition}
\begin{proof}
  We first estimate
  \begin{equation}
    \label{eq:que-jnorm'}
    \bigabs{\norm{Jf}-\norm f}
    =\frac{\bigabs{\normsqr{Jf}-\normsqr f}}{\norm{Jf}+ \norm f}
    =\frac{\bigabs{\iprod{J^*Jf-f} f}} {\norm{Jf}+ \norm f}
    \leCS \frac{\delta \energy(f)^{1/2}\norm f} {\norm f}
    = \delta \energy(f)^{1/2}
  \end{equation}
  for $f \in \HS^1$ using~\eqref{eq:quasi-uni.b}; this implies
  \begin{equation*}
    \norm{J^1f}-\norm f
    \ge \bigl(\norm{Jf}- \norm f\bigr) - \norm{(J-J^1)f}
    \ge -2\delta \energy(f)^{1/2}
  \end{equation*}
  using the previous estimate and~\eqref{eq:quasi-uni.c}.  We conclude
  that
  \begin{equation*}
    \normsqr{J^1f}-\normsqr f
    = \bigl(\norm{J^1f}-\norm f\bigr)\bigl(\norm{J^1f}+\norm f\bigr)
    \ge -2\delta \energy(f)^{1/2} \norm f
    \geCY -\delta \bigl(\normsqr f + \energy(f)\bigr).
  \end{equation*}
  For the energy forms, we have
  \begin{equation*}
    \wt \energy(J^1f)-\energy(f)
    =\bigl(\wt \energy(J^1f,J^1f)-\energy(f,J^{\prime 1} J^1f)\bigr)
    +\energy(f,J^{\prime 1} J^1f-f)
    \le 0
  \end{equation*}
  using~\eqref{eq:quasi-uni.d} (with $\delta=0$)
  and~\eqref{eq:q-u-e.ew}.  The result follows from \Prp{ew.comp}.

  The second estimate can be seen similarly by swapping $\energy$ and
  $\wt \energy$, using the identification operators in the opposite
  direction and~\eqref{eq:q-u-e.ew'}
\end{proof}

%
\section{
  Post-critically finite fractals and self-similar
  energy forms}
\label{sec:fractals}
%
In this section, we briefly review some facts on fractals and
self-similar energy forms, see~\cite{kigami:01}
or~\cite[Ch.~4]{strichartz:06}.  Let $(X,d)$ be a complete metric
space.  Typically, we work with $X=\R^d$ with the Euclidean metric
(see \Sec{examples}), but our abstract convergence results also hold
in the general case.

\subsection{Post-critically finite self-similar fractals and their
  approximating graphs}
We start with a finite family $F=(F_j)_{j=1,\dots,N}$ of contractive
similarities $\map{F_j}X X$, i.e.,
\begin{equation*}
  d(F_j(x),F_j(y))=\theta_jd(x,y)
\end{equation*}
for all $x,y \in X$, where $\theta_j \in (0,1)$ denotes the
\emph{contraction ratio}.  By the Banach fixed point theorem, such a
similarity has a unique \emph{fixed point} $q_j \in X$.  The family
$F$ is called an \emph{iterated function system (IFS)}.  For any IFS,
there exists a unique (non-empty) compact subset $K \subset X$, called
\emph{self-similar set} or \emph{self-similar fractal}, such that $K$
is self-similar with respect to $F$, i.e., such
that~\eqref{eq:similarities} holds.  The existence is guaranteed by
Banach's fixed point theorem on the space of compact subsets of $X$
with Hausdorff distance (see~\cite[Thm.~1.1.7]{kigami:01}).

For a word $w=w_1\dots w_m$ of length $|w|=m$ over the alphabet
$\{1,\dots,N\}$ we define the map $F_w$ by $F_w=F_{w_1}\circ\dots\circ
F_{w_m}$. We denote by $W_m=\{1,\dots,N\}^m$ the set of all words of
length $m$.  Then there is a natural cell structure on the
self-similar set $K$ described by the map $W_m\ni w\mapsto F_w(K)$ and
we call $K_w:=F_w(K)$ an \emph{$m$-cell}.

\begin{definition}
  \label{def:pcf}
  We say that a self-similar set $K$ is \emph{post-critically finite}
  (or \emph \pcf for short) if $K$ is connected, and if there exists
  a finite set $V_0 \subset K$, called the \emph{boundary} of $K$,
  such that
  \begin{equation*}
    F_w(K) \cap F_{w'}(K) 
    \subset F_w(V_0) \cap F_{w'}(V_0)
  \end{equation*}
  for all words $w, w'$ ($w \ne w'$) of the same length.  Moreover, we
  assume that each boundary point is a fixed point of the IFS, i.e.,
  that $V_0 \subset \{q_1,\dots,q_N\}$.
\end{definition}
In other words, on a \pcf self-similar fractal, its $m$-cells
intersect only in the corresponding $m$-cell boundary points, hence in
a finite set of at most $N$ elements.

For $m=0$, we let $G_0=(V_0,E_0)$ be the complete graph.  For $m \in
\N$, we denote by
\begin{equation*}
  V_m
  := \bigcup_{w \in W_m} F_w(V_0)
\end{equation*}
the union of all \emph{$m$-cell boundary points} and
$V_0$. 
We consider $V_m$ as the vertex set of a graph $G_m$ with edges
\begin{equation*}
  E_m := 
  \bigset{\{x,y\} \subset V_m}
      {\text{$x \ne y$ and there is $w \in W_m$ such that $x,y \in F_w(V_0)$}}.
\end{equation*}
We also write $x \sim_m y$ if $\{x,y\} \in E_m$.  We call $x \in V_m$
a \emph{junction point} if $x$ lies in more than one $m$-cell.  It is
clear that once, the IFS and the graph $G_0=(V_0,E_0)$ at level $m=0$
are given, the sequence of graphs $G_m=(V_m,E_m)$ is recursively
defined.  We call $(G_m)_{m \in \N_0}$ the \emph{approximating graph
  sequence} associated with the self-similar set $K$ and its boundary
$V_0$.

\subsection{Energy forms on the approximating graphs}
Associated with the graph $G_m$ is its \emph{energy}, given by the
quadratic form
\begin{equation}
  \label{eq:graph.energy'}
  \energy_{G_m}(f)
  = \sum_{ \{x,y\} \in E_m} \conductance_{\{x,y\},m} \abssqr{f(x)-f(y)}
\end{equation}
for functions $f \in \lspace{V_m}:=\{\map f {V_m} \C\}$; we specify
the \emph{conductances} $\conductance_{\{x,y\},m}>0$ whenever $x \sim_m y$
later on. 

\begin{definition}[{\cite[Def.~2.2.1]{kigami:01}}]
  \label{def:compatible}
  We say that a sequence $(\energy_{G_m})_{m \in \N_0}$ of energy forms on
  a sequence of graphs $(G_m)_{m \in \N_0}$ is \emph{compatible}, if
  the vertex sets of $G_m$ are nested (i.e., $V_m \subset V_{m+1}$)
  and if
  \begin{equation}
    \label{eq:compatible}
    \energy_{G_m}(\phi)=
    \min\bigset{\energy_{G_{m+1}}(f)}
    {f \in \lspace{V_{m+1}}, \; f \restr {V_m}=\phi}
  \end{equation}
  for all $\phi \in \lspace{V_m}$.  The (unique) minimiser $h_{m+1}
  \in \lspace{V_{m+1}}$ of~\eqref{eq:compatible} is called
  \emph{harmonic extension} of $\phi \in \lspace{V_m}$.
\end{definition}
\begin{remark*}
  The right hand side of~\eqref{eq:compatible} can be seen as a
  Dirichlet-to-Neumann form of the graph $V_{m+1}$ with boundary
  $V_m$; compatibility then means that the $(m+1)$-th
  Dirichlet-to-Neumann form actually agrees with the $m$-th graph
  energy $\energy_{G_m}$ (see e.g.~\cite[Sec.~2.1]{kigami:01}
  or~\cite[Sec.~6.7]{post:16} for this interpretation).
\end{remark*}
For a compatible sequence $(\energy_{G_m})_m$ we can define a limit
form (see~\cite[Sec.~2.2]{kigami:01}).  Let $V_*:=\bigcup_{m \in \N_0}
V_m$ and $u \in \lspace{V_*}$.  As $u \restr{V_{m+1}}$ is an extension
of $u \restr {V_m}$, we have
\begin{equation*}
  \energy_{G_m}(u \restr {V_m})
  \le \energy_{G_{m+1}}(u \restr {V_{m+1}})
\end{equation*}
by~\eqref{eq:compatible}, hence
\begin{equation}
  \label{eq:limit.form}
  \energy_K(u) := \lim_{m \to \infty} \energy_{G_m}(u \restr {V_m})
\end{equation}
exists (and may be $\infty$).

\begin{definition}
  \label{def:forms.self-similar}
  Let $(G_m)_{m \in \N_0 }$ be the approximating sequence of graphs
  associated with the IFS $(F_j)_{j=1,\dots,N}$, the corresponding
  self-similar set $K$ and its boundary $V_0$.  We call a sequence
  $(\energy_{G_m})_{m \in \N_0}$ of energy forms on $(G_m)_m$
  \emph{self-similar} if there exist so-called \emph{renormalisation
    factors} $r_j \in (0,1)$ for $j=1,\dots,N$ such that
  \begin{equation}
    \label{eq:self-similar1}
    \energy_{G_{m+1}}(u)
    =\sum_{j=1}^N r_j^{-1} \energy_{G_m}(u\circ F_j)
  \end{equation}
  for $u \in \lspace{V_{m+1}}$.
\end{definition}
Clearly, a self-similar sequence $(\energy_{G_m})_m$ can be defined
recursively, given $\energy_{G_0}$ on $\lspace{V_0}$
by~\eqref{eq:self-similar1}.  As a result, we then obtain
\begin{equation}
  \label{eq:self-similar2}
  \energy_{G_m}(u)
  = \sum_{w \in W_m} r_w^{-1} \energy_{G_0}(u \circ F_w)
\end{equation}
where $r_w:=r_{w_1}\cdot \ldots \cdot r_{w_m}$.  In particular, the
conductances are given by
\begin{equation}
  \label{eq:tau.est}
  \conductance_{e,m} = \frac {\conductance_{e_0,0}}{r_w} 
  \in \bigl[\conductance_{-,m},\conductance_{+,m} \bigr]
  \quadtext{with}
  \conductance_{\pm,m}
  :=\frac{\conductance_{\pm,0}}{(r_\mp)^m},
\end{equation}
where $w \in W_m$ is given by $e=F_w(e_0)$.  Moreover, $r_\pm$ is the
maximal resp.\ minimal value of the renormalisation factors
$r_1,\dots,r_N$, and similarly $\conductance_{\pm,0}$ is the maximal
resp.\ minimal value of the conductances $\conductance_{e_0,0}$ of
$\energy_{G_0}$.

We now make our main assumption on the fractal:
\begin{definition}
  \label{def:renorm}
  We say that a fractal $K$ is \emph{approximable by finite weighted
    graphs}, if $K$ is a \pcf self-similar set given by an IFS, if
  there is an approximating sequence of finite weighted graphs
  $(G_m)_{m \in \N_0}$, and if there is a \emph{compatible and
    self-similar} sequence $(\energy_{G_m})_m$ of energy forms
  $\energy_{G_m}$ on $G_m$.  We call $V_0$ the \emph{boundary of $K$}.
\end{definition}
On a fractal $K$ approximable by finite weighted graphs, there exists
a natural energy form $\energy_K$ defined by~\eqref{eq:limit.form}.
Note that such an energy form is uniquely determined by the data
$((F_j)_{j=1,\dots,N},(r_j)_{j=1,\dots,N},\energy_{G_0})$.

The \pcf property guarantees the existence of the approximating
sequence of graphs.  Moreover, given an energy form $\energy_{G_0}$
and renormalisation factors $r_1,\dots,r_N$, one can define a
self-similar sequence of energy forms $(\energy_{G_m})_m$
by~\eqref{eq:self-similar2}.

\begin{remark}
  \label{rem:renorm.prob}
  The difficult problem is to find $\energy_{G_0}$ (i.e., conductances
  $(\conductance_{e_0,0})_{e_0 \in E_0}$) and $r_1,\dots,r_N$ such
  that $(\energy_{G_m})_m$ (defined via~\eqref{eq:self-similar2} and
  hence self-similar) is at the same time \emph{compatible}.  This
  problem is called the \emph{renormalisation problem}, and can be
  rephrased as a fixed point problem or a non-linear eigenvalue
  problem.  It can be shown (see e.g.~\cite[Prop.~3.1.3]{kigami:01})
  that $(\energy_{G_m})_m$ defined via~\eqref{eq:self-similar2} is
  compatible if and only if it is compatible at level $m=0$
  (i.e.,~\eqref{eq:compatible} holds for $m=0$).  The renormalisation
  problem is not (yet) solved for general \pcf fractals (see
  e.g.~\cite[Sec.4.2, p.~98f]{strichartz:06}), but there are many
  examples (see \Sec{examples}).
\end{remark}

\subsection{The energy form on the fractal}

Given now a fractal $K$ approximable by finite weighted graphs, we can
define a limit form $\energy_K$ as in~\eqref{eq:limit.form} and the
self-similarity~\eqref{eq:self-similar1} survives the limit:
$\energy_K$ is \emph{self-similar}, i.e.,
\begin{equation}
  \label{eq:self-similar3}
  \energy_K(u)
  =\sum_{j=1}^N r_j^{-1} \energy_K(u\circ F_j)
\end{equation}
holds for all $u \in \lspace{V_*}$.  Let now $\energy_K$
be the form defined in~\eqref{eq:limit.form} on
\begin{equation*}
  \dom \energy_K
  := \set{u \in \Cont K}
  {\energy_K(u):=\lim_{m \to \infty} \energy_{G_m}(u \restr {V_m}) < \infty}.
\end{equation*}
As $V_*=\bigcup_m V_m$ is dense in $K$ and as a continuous function
on the compact set $K$ is uniformly continuous, $u$ is indeed determined
by its values on $V_*$.

The compatibility of the sequence $(\energy_{G_m})_m$ now passes over to
the limit in the following sense (see~\cite[Lem.~2.2.2]{kigami:01}):
for any ``boundary value'' $\phi \in \lspace{V_m}$ there exists a
unique continuous function $h \in \dom \energy_K$ on $K$ such
that $h \restr{V_m}=\phi$ and
\begin{equation}
  \label{eq:harm.ext}
  \energy_{G_m}(\phi)=
  \energy_K(h)=
  \min \bigset {\energy_K(u)} {u \in \dom \energy_K, u \restr{V_m}=\phi}.
\end{equation}
These functions are called \emph{$m$-harmonic functions}. In the
special case where $\phi$ is the characteristic function $\1_x$ of the
set $\{x\}$ for $x\in V_m$ we denote the $m$-harmonic function with
boundary value $\1_x$ by $\psi_{x,m}$.  By polarisation and a simple
argument, it follows for the corresponding sesquilinear forms that
\begin{equation}
  \label{eq:harmonic}
  \energy_K(h, u)
  = \energy_{G_m}(h \restr{V_m}, u \restr{V_m})
\end{equation}
for all $u \in \dom \energy_K$ and $h \in \dom \energy_K$ an $m$-harmonic
function.

From~\eqref{eq:self-similar3} we derive the following ``localisation''
of the energy form $\energy_K$.  For any $w\in W_m$ we define
$\energy_{K_w}$ with $\dom \energy_{K_w}=\set{u \restr{K_w}}{u \in
  \dom \energy_K}$ by
\begin{equation}
  \label{eq:def.local.energy}
  \energy_{K_w}(u \restr{K_w})
  :=r_w^{-1} \energy_K(u \restr{K_w} \circ F_w)
\end{equation}
for all $u \in \dom \energy_K$.
Then we have the following useful ``localisation'' formula
\begin{equation}
  \label{eq:local.energy}
  \energy_K(u) 
  = \sum_{w \in W_m} \energy_{K_w}(u \restr{K_w})
\end{equation}
for all $u \in \dom \energy_K$.

As each cell $K_w$ described by the IFS $\{F_{wj}\}_{j=1,\dots,N}$ is
itself a pcf self-similar set, we conclude from the definition of the
resistance metric (see~\cite[Secs.~2.3]{kigami:01} or~\cite[Sec.~4.4
and~1.6]{strichartz:06}):
\begin{proposition}
  \label{prp:hoelder}
  Let $K$ be a fractal approximable by finite weighted graphs with
  self-similar energy form $\energy_K$ and let $x,y \in K_w$ be in a
  cell $K_w$ of generation $m$.  Then
  \begin{equation}
    \label{eq:hoelder-ineq}
    \bigabssqr{u(x)-u(y)}
    \le \frac {r_w}{\conductance_{-,0}} \energy_{K_w} (u \restr{K_w})
    \le \frac 1 {\conductance_{-,m}} \energy_{K_w} (u \restr{K_w})
  \end{equation}
  for all $u \in \dom \energy_K$.
\end{proposition}

\begin{remark*}
  Note that the localisation of $\energy_K$ in (\ref{eq:local.energy})
  allows us to consider only the energy on $K_w$ instead of $K$ on
  the right hand side.  This slight detail will help us to obtain the
  optimal estimate in the second estimates of~\eqref{eq:quasi-uni.b}
  and~\eqref{eq:quasi-uni.c} in the proof of \Thm{main} in
  \Subsec{approx.fractals.graphs}.
\end{remark*}

\subsection{Measures on fractals and graphs and the associated
  operators}

In order to have a \emph{closed} form $\energy_K$ and in order to
define an associated self-adjoint operator, we need a Hilbert space
structure on $K$, i.e., a measure $\mu$ on $K$.  We assume that $\mu$
is a Borel regular probability measure of full support (see e.g.
~\cite[Sec.~1.4]{kigami:01}).  Examples such as a \emph{self-similar}
measure are given in \Subsec{ex.measure}.  The Hilbert space is then
$\Lsqr{K,\mu}$ with its usual inner product.

The domain $\dom \energy_K$ of $\energy_K$ is itself a Hilbert space,
and it can be shown that the embedding $\dom \energy_K \subset
\Lsqr{K,\mu}$ is compact.  This implies that the operator
$\Delta_{(K,\mu)}$ associated with the quadratic form $\energy_K$,
i.e., the operator defined via
\begin{equation*}
  \energy_K(u,v) = \iprod[\Lsqr{K,\mu}] {\Delta_{(K,\mu)} u} v
\end{equation*}
for all $v \in \dom \energy_K$, has compact resolvent, hence purely
discrete spectrum (note that --- in contrast to some literature like
Kigami's or Strichartz's books~\cite{kigami:01,strichartz:06} --- our
operator is non-\emph{negative}, i.e, $\Delta_{(K,\mu)}\ \ge 0$).
Moreover, it can be shown that $\dom \energy_K \subset \Cont K$, and
it makes sense to evaluate $u \in \dom \energy_K$ at points of $K$.

For the Hilbert space structure on the graphs $G_m$, we also need a
measure $\mu_m$ on $V_m$.  We set
\begin{equation}
  \label{eq:def.mu.m}
  \mu_m(x)=\int_K\psi_{x,m}\dd\mu
\end{equation}
and we call the sequence $(\mu_m)_m$ \emph{approximating measures}
corresponding to $(K,\mu)$.  The measures $\mu_m$ are actually
discrete probability measures of full support because the family of
$m$-harmonic functions $\{\psi_{x,m}\}_{x \in V_m}$ forms a partition
of unity and $\mu$ is a probability measure of full support on $K$.

By $\lsqr{V_m,\mu_m}$ we denote the finite dimensional Hilbert space
$\lspace{V_m}$ with norm (and hence inner product) given by
\begin{equation*}
  \normsqr[\lsqr{V_m,\mu_m}] f
  := \sum_{x \in V_m} \mu_m(x) \abssqr{f(x)}.
\end{equation*}
A simple calculation shows that the (bounded) operator
$\Delta_{(G_m,\mu_m)} \ge 0$ associated with the energy form
$\energy_{G_m}$ acts as
\begin{equation*}
  (\Delta_{(G_m,\mu_m)} f)(x)
  = \frac 1{\mu_m(x)} \sum_{y \sim_m x} \conductance_{\{x,y\},m}
      \bigl(f(x)-f(y)\bigr).
\end{equation*}

%
\section{
  Proof of the quasi-unitary equivalence}
\label{sec:approx.fractals.graphs}
%

\subsection{Quasi-unitary equivalence of graph and fractal energy
  forms}
\label{ssec:approx.fractals.graphs}

We come now to the main part of our article, the proof of \Thm{main}:
Let $\HS_m=\lsqr{V_m,\mu_m}$ and let $\energy_m=\energy_{G_m}$ be the
graph energy form as in~\eqref{eq:graph.energy}
or~\eqref{eq:graph.energy'}.  For the limit space, we set $\wt
\HS=\Lsqr {K,\mu}$ and $\wt \HS^1 = \dom \energy_K$, where $\wt
\energy=\energy_K$ denotes the energy form on the fractal, defined as
a limit in~\eqref{eq:limit.form}.

We will now show our first main result \Thm{main}:
\begin{proof}[Proof of \Thm{main}]
  We define the identification operators $\map{J=J_m}{\HS_m} {\wt
    \HS}$ and $J'=J'_m$ by
  \begin{equation*}
    Jf
    :=\sum_{x\in V_m}f(x) \psi_{x,m} \in \wt \HS,
    \qquad 
    (J'u)(y):=\frac{1}{\mu_m(y)}\iprod {u} {\psi_{y,m}}
  \end{equation*}
  for $y \in V_m$, then it is easily seen that $J^*=J'$.  Moreover, we
  let $J^1f :=Jf$ for $f \in \HS_m$ and let
  \begin{equation*}
    \map{J^{\prime1}}{\wt \HS^1}{\HS_m^1}
    \quadtext{be given by}
    (J^{\prime 1}u)(x)=u(x),
  \end{equation*}
  the evaluation of $u$ in $x \in V_m \subset K$.

  Note that we have $x \sim_m y$ if and only if $\iprod
  {\psi_{x,m}}{\psi_{y,m}}\neq 0$ for $x,y\in V_m$ ($x \ne y$), where
  $\psi_{x,m}$ is the $m$-harmonic function with value $1$ at $x \in
  V_m$ and $0$ elsewhere.  We will make frequent use of the following
  relation,
  \begin{equation}
    \label{eq:trick}
    \mu_m(y) =\int_K\psi_{y,m}\dd\mu
    = \sum_{x \in V_m} \iprod {\psi_{x,m}}{\psi_{y,m}},
  \end{equation}
  which follows from the fact that $\sum_{y \in V_m} \psi_{y,m}=\1$.
  By the Cauchy-Young inequality and~\eqref{eq:trick}, we have
  \begin{align*}
    \normsqr[\Lsqr{K,\mu}] {Jf}
    &= \sum_{x,y \in V_m} f(x)\conj{f(y)} 
        \underbrace{\iprod[\Lsqr{K,\mu}]{\psi_{x,m}}{\psi_{y,m}}}_{\ge 0}\\
    &\leCY \frac 12 
       \sum_{x \in V_m} \abssqr{f(x)} \sum_{y \in V_m} \iprod{\psi_{x,m}}{\psi_{y,m}}
      +   \frac 12 
       \sum_{y \in V_m} \abssqr{f(y)} \sum_{x \in V_m} \iprod{\psi_{x,m}}{\psi_{y,m}}\\
    &= \sum_{x \in V_m} \abssqr{f(x)} \underbrace{\iprod{\psi_{x,m}}{\1}}_{=\mu_m(x)}
     = \normsqr[\lsqr{V_m,\mu_m}] f,
  \end{align*}
  hence we have chosen our vertex weights $\mu_m$ on $V_m$
  appropriately; in particular, the first
  condition~\eqref{eq:quasi-uni.a} of quasi-unitary equivalence is
  fulfilled as well as the second (recall that $J^*=J'$), both with
  $\delta=0$.

  Let us now check the first estimate in~\eqref{eq:quasi-uni.b}: Let
  $y\in V_m$. We apply~\eqref{eq:trick} and obtain
  \begin{align*}
    (f-J'Jf)(y)&
    =\frac{1}{\mu_m(y)}\sum_{x \in V_m}\iprod
    {\psi_{x,m}}{\psi_{y,m}}(f(y)-f(x)).
  \end{align*}
  Hence, for $f\in\HS_m$, by the Cauchy-Schwarz inequality,
  \begin{align*}
    \normsqr[\lsqr{V_m,\mu_m}]{f-J'Jf}
    &=\sum_{y\in V_m}\frac{1}{\mu_m(y)}
      \Bigabssqr{\sum_{x \in V_m} \iprod {\psi_{x,m}}{\psi_{y,m}} (f(y)-f(x))}\\
    &\leCS \sum_{y\in V_m} \frac 1{\mu_m(y)} 
       \bigg(\sum_{x \in V_m}
         \frac{\iprod {\psi_{x,m}}{\psi_{y,m}}^2}{\conductance_{\{x,y\},m}}\bigg)
           \sum_{x \sim_m y} \conductance_{\{x,y\},m} \abssqr{f(y)-f(x)}\\
    &\le \frac 1 {\conductance_{-,m}} 
    \bigg(\max_{y\in V_m}\frac 1{\mu_m(y)}\sum_{x \in V_m}
            \iprod {\psi_{x,m}}{\psi_{y,m}}^2
    \bigg) \energy_m(f)
  \end{align*}
  using~\eqref{eq:tau.est} for the last estimate.  Now
  \begin{equation*}
    \max_{y\in V_m}\frac 1{\mu_m(y)}\sum_{x \in V_m} 
                  \iprod {\psi_{x,m}}{\psi_{y,m}}^2
    \le \max_{y\in V_m}  \sum_{x \in V_m} \iprod {\psi_{x,m}}{\psi_{y,m}}
    = \max_{y \in V_m} \mu_m(y)
    \le \mu_{+,m}
  \end{equation*}
  using $\iprod{\psi_{x,m}}{\psi_{y,m}} \le \iprod{\1}{\psi_{y,m}} =
  \mu_m(y)$ for the first inequality, \eqref{eq:trick} for the
  equality and $\psi_{y,m}\le 1$ for the last inequality.  Moreover,
  we have estimated $\mu_m(y)$ by
  \begin{equation}
    \label{eq:mu.m.pm}
    \mu_{+,m} 
    := \max \bigl\{
    \max_{x \in V_m} \mu_m(x),
    \max_{w \in W_m} \mu(K_w)
    \bigr\}
  \end{equation}
  (the reason for the second term will become clear in a moment).  In
  particular, the first estimate in~\eqref{eq:quasi-uni.b} holds with
  $\delta=(\mu_{+,m}/\conductance_{-,m})^{1/2}$.

  Next we prove the second inequality of~\eqref{eq:quasi-uni.b};
  actually, we will show estimate~\eqref{eq:quasi-uni.b''}: We have
  \begin{equation*}
    u-JJ^{\prime 1}u
    =\sum_{x\in V_m} (u-u(x)) \psi_{x,m}
    =\sum_{w \in W_m} \sum_{x\in F_w(V_0)} (u-u(x)) \psi_{x,m}\restr{K_w}
  \end{equation*}
  almost everywhere for any $u\in\HS$ using the fact that
  $\{\psi_{x,m}\}_{x\in V_m}$ is a partition of unity.
  In particular, we have
  \begin{align*}
    \normsqr[\Lsqr{K,\mu}]{u-JJ^{\prime 1}u}
    &= \sum_{w \in W_m } \normsqr[\Lsqr{K_w}]{u-JJ^{\prime 1}u}\\
    &\le
     \sum_{w \in W_m} \sum_{x,y \in F_w(V_0)} 
       \iprod[\Lsqr{K_w}]{\psi_{x,m}}{\psi_{y,m}} 
              \max_{x,z \in K_w}\abs{u(x)-u(z)}
              \max_{y,z \in K_w}\abs{u(y)-u(z)}\\
    &\le \frac 1{\conductance_{-,m}}  
     \sum_{w \in W_m} \sum_{x,y \in F_w(V_0)} 
       \iprod[\Lsqr{K_w}]{\psi_{x,m}}{\psi_{y,m}} 
              \energy_{K_w}(u \restr{K_w})\\
    &= \frac 1{\conductance_{-,m}}  
     \sum_{w \in W_m} \mu(K_w) \energy_{K_w}(u \restr{K_w})
    \le \frac {\mu_{+,m}}{\conductance_{-,m}}  \wt\energy(u)
  \end{align*}
  using \Prp{hoelder} for the second inequality, $\sum_{x \in
    F_w(V_0)} \psi_{x,m} \restr{K_w}=\1_{K_w}$ for the fourth line and
  (\ref{eq:local.energy}) resp.~\eqref{eq:mu.m.pm} for the final
  inequality.  In particular, we can again choose
  $\delta'=(\mu_{+,m}/\conductance_{-,m})^{1/2}$
  in~\eqref{eq:quasi-uni.b''}.

  For the second estimate in~\eqref{eq:quasi-uni.c} (the first one is
  trivially fulfilled), let $u \in \HS$ and $x\in V_m$.  Then
  \begin{align*}
    (J'u-J^{\prime 1}u)(x)
    =\frac{1}{\mu_m(x)}\bigiprod {u-u(x)\1}{\psi_{x,m}}
    =\sum_{w \in W_{x,m}} \frac 1{\mu_m(x)}
         \bigiprod[\Lsqr{K_w}] {u-u(x)\1}{\psi_{x,m}},
  \end{align*}
  where $W_{x,m}:=\set{w \in W_m}{x \in K_w}$ are the words whose
  $m$-cells contain $x$.  Hence, we have
  \begin{align*}
    \bigabssqr{(J'u-J^{\prime 1}u)(x)}
    &\le \Bigl(
           \sum_{w \in W_{x,m}} \frac 1{\mu_m(x)}
              \int_{K_w} \abs{u-u(x)} \psi_{x,m} \dd \mu
         \Bigr)^2\\
    &\le \frac 1{\conductance_{-,m}}\max_{w \in W_{x,m}}  
            \energy_{K_w}\bigl(u \restr{K_w}\bigr) 
         \Bigl(
           \frac 1{\mu_m(x)} \sum_{w \in W_{x,m}} 
              \int_{K_w} \psi_{x,m} \dd \mu
         \Bigr)^2\\
    &=  \frac 1{\conductance_{-,m}}\max_{w \in W_{x,m}}  
            \energy_{K_w}\bigl(u \restr{K_w}\bigr)
      \le  \frac 1{\conductance_{-,m}}\sum_{w \in W_{x,m}}  
            \energy_{K_w}\bigl(u \restr{K_w}\bigr)
  \end{align*}
  using \Prp{hoelder} for the second inequality, $\sum_{w \in W_{x,m}}
  \int_{K_w} \psi_{x,m} \dd \mu=\mu_m(x)$ for the last line.  Now, we
  obtain for the norm estimate
  \begin{align*}
    \normsqr[\lsqr{V_m,\mu_m}]{J'u-J^{\prime 1}u}
     &\le \frac 1{\conductance_{-,m}} 
       \sum_{x \in V_m} \sum_{w \in W_{x,m}} 
            \energy_{K_w}\bigl(u \restr{K_w}\bigr) \mu_m(x)\\
     &=  \frac 1{\conductance_{-,m}} 
       \sum_{w \in W_m} \energy_{K_w} \bigl(u \restr{K_w}\bigr)
           \sum_{x \in F_w(V_0)}\mu_m(x)\\
     &\le \frac {N_0 \mu_{+,m}} {\conductance_{-,m}} 
       \sum_{w \in W_m} 
         \energy_{K_w}(u \restr{K_w})
     = \frac {N_0\mu_{+,m}}{\conductance_{-,m}} \wt \energy(u)
  \end{align*}
  using~\eqref{eq:mu.m.pm} resp.\ (\ref{eq:local.energy}) for the last line
  (here $N_0=\card{V_0}$ denotes the number of boundary vertices).

  Finally, we check the last condition~\eqref{eq:quasi-uni.d} of
  quasi-unitary equivalence: For any $f=\sum_{x\in V_m}f(x)\psi_{x,m}
  \restr{V_m} \in\HS_m$ and $u\in\HS^1$, we have
  \begin{equation*}
    \energy_m(f,J^{\prime 1}u) - \wt \energy(J^1f,u)
    = \sum_{x\in V_m} f(x) 
         \bigl(\energy_m(\psi_{x,m} \restr{V_m},u \restr{V_m})
              -\wt \energy(\psi_{x,m},u)\bigr)
    =0
  \end{equation*}
  using~\eqref{eq:harmonic}.  We now apply \Lem{quasi-uni.b} and
  obtain~\eqref{eq:quasi-uni.b} with $\deltaA=0$,
  $\delta'=(\mu_{+,m}/\conductance_{-,m})^{1/2}$ and
  $\deltaC=(N_0\mu_{+,m}/\conductance_{-,m})^{1/2}$.  Then
  $\delta=\delta'+\deltaC=(1+\sqrt{N_0})(\mu_{+,m}/\conductance_{-,m})^{1/2}$.
  Collecting all the individual error terms, the quasi-unitary
  equivalence constant is then
  \begin{equation}
    \label{eq:main.delta}
    \delta_m 
    = (1+\sqrt{N_0}) \cdot
        \Bigl(\frac {\mu_{+,m}}{\conductance_{-,m}}\Bigr)^{1/2}
    = \frac{1+\sqrt{N_0}}{\sqrt{\conductance_{-,0}}} \cdot
          (r_+^m \mu_{+,m})^{1/2}.
    \qedhere
  \end{equation}
\end{proof}
Note that for a general probability measure $\mu$ (not necessarily
self-similar), we have at least the estimate $\mu_{+,m}\le 1$ and
hence the following result:

\begin{corollary}
  \label{cor:main}
  The $\delta_m$-quasi-unitary equivalence of the fractal energy form
  $\energy_K$ with general Borel regular probability measure $\mu$ of
  full support and the approximating graph energy forms
  $\energy_m=\energy_{G_m}$ as in \Thm{main} holds with
  \begin{equation*}
    \delta_m 
    = \frac{1+\sqrt{N_0}}{\sqrt{\conductance_{-,0}}} 
      \cdot r_+^{m/2}.
  \end{equation*}
\end{corollary}
If the measure is self-similar or if the fractal, its measure and
boundary are symmetric, we can obtain better estimates, see
\Subsec{ex.symmetric}.

\begin{remark}
  \label{rem:fin.ram.frac}
  Our results also extends to certain classes of finitely ramified
  fractals as introduced in~\cite{teplyaev:08}: if there is a
  compatible sequence of energy (or more precisely, resistance) forms
  on the boundary structure $(V_m)_m$.  If the diameter of $m$-cells
  tends to $0$ in the resistance metric, and if one chooses a measure
  such that the measure of $m$-cells tends to $0$, then we obtain a
  similar result as \Thm{main}.
\end{remark}

\begin{remark}
  \label{rem:fem}
  Let us comment on the finite element method (FEM) for fractals
  developed in~\cite{grs:01,asst:01}.  In our notation, one can find
  for example an approximative eigenvalue $\lambda_m$ of an
  eigenvalue $\wt \lambda$ of $\Delta_{(K,\mu)}$ by finding a
  non-trivial solution $f \in \lsqr{V_m,\mu_m}$ (or a vector $f \in
  \C^{V_m}$) of the generalised eigenvalue problem
  \begin{equation*}
    C_m f = \lambda_m G_m f,
  \end{equation*}
  where the matrices $C_m$ and $G_m$ are given by
  \begin{equation*}
    (C_m)_{xy}=\wt \energy(\psi_{x,m},\psi_{y,m})
    =\energy_m(\delta_x,\delta_y)
    =
    \begin{cases}
      \sum_{z \sim_m x} \conductance_{\{x,z\},m}, &\text{if $x = y$}\\
      -\conductance_{\{x,y\},m}, &\text{if $x \sim_m y$}\\
      0, & \text{otherwise}
    \end{cases}
  \end{equation*}
  and
  \begin{equation*}
    (G_m)_{xy}=\iprod[\wt\HS]{\psi_{x,m}}{\psi_{y,m}}.
  \end{equation*}
  The latter matrix is also called the \emph{Gram matrix} of
  $(\psi_{x,m})_{x \in V_m}$.  As a candidate for $f$, one can choose
  e.g.\ evaluation of $u$ at the vertices $V_m$, i.e., $f=J^{\prime 1}
  u$.  Then $f$ is the coefficient vector of the harmonic
  interpolation spline $JJ^{\prime 1}u=\sum_{x \in V_m} f(x)
  \psi_{x,m}$ in our notation.  We show actually as one condition of
  quasi-unitary equivalence that the harmonic spline $Jf=JJ^{\prime
    1}u$ is close to the original $u$ for large $m \in \N$, see
  \Lem{quasi-uni.b}.
\end{remark}

\subsection{A direct eigenvalue convergence result}
\label{ssec:ev.direct}
Although an eigenvalue convergence already follows from quasi-unitary
equivalence as in \Cor{main2}, we show a more explicit convergence
using \Prp{q-u-e.ew}:
\begin{theorem}
  \label{thm:ew.direct}
  Let $(K,\mu)$ be a \pcf fractal with Borel regular probability
  measure of full support approximable by finite weighted graphs
  $(G_m,\mu_m)$ and let $\lambda_k(G_m,\mu_m)$ resp.\ $\lambda_k(K,\mu)$
  be the $k$-th eigenvalue of the Laplacian operator associated with
  $(G_m,\mu_m)$ resp.\ $(K,\mu)$ (in increasing order, repeated
  according to their multiplicity).  Then
  \begin{equation}
    \label{eq:ew.direct1}
    \frac{1-\delta_m}{1+\delta_m\lambda_{k,+}} \cdot
      \lambda_k(G_m,\mu_m)
    \le \lambda_k(K,\mu)
    \le \frac 1 {1-\delta_m(1+\lambda_{k,+})} \cdot 
      \lambda_k(G_m,\mu_m)
  \end{equation}
  provided $m \ge m_0$, where $\delta_m$ is as in \Thm{main}
  resp.~\eqref{eq:main.delta}.  Here, $m_0 \in \N$ is defined by the
  condition $\delta_{m_0} < (1+\lambda_{k,+})^{-1}$, where
  $\lambda_k(G_m,\mu_m)\le \lambda_{k,+}$ is an upper bound on the
  graph eigenvalues.

  If $\delta_m \le (1+\lambda_k(K,\mu))^{-1}/2$, then we can choose as
  upper bound
  \begin{equation}
    \label{eq:ew.direct2}
    \lambda_k(G_m,\mu_m)
    \le 2 \lambda_k(K,\mu) =: \lambda_{k,+}.
  \end{equation}
\end{theorem}
\begin{proof}
  We apply \Prp{q-u-e.ew} with $\HS=\lsqr{V_m,\mu_m}$ and $\wt
  \HS=\Lsqr{K,\mu}$ etc.  We note first that~\eqref{eq:q-u-e.ew} is
  trivially fulfilled (even with equality) as
  \begin{equation*}
    (J^{\prime 1}J^1 f)(y)
    =\sum_{x \in V_m}f(x)\psi_{x,m}(y)=f(y).
  \end{equation*}
  In particular, the upper estimate on $\lambda_k(K,\mu)$
  in~\eqref{eq:ew.direct1} holds (using the second inequality
  in~\eqref{eq:q-u-e.ew.result}).

  For the lower bound~\eqref{eq:ew.direct1}, we note
  that~\eqref{eq:q-u-e.ew'} is fulfilled as
  \begin{equation*}
    \wt \energy(u,J^1J^{\prime 1} u)
    = \energy_m (u \restr {V_m}, (J^1J^{\prime 1} u)\restr{V_m})
    = \energy_m (u \restr {V_m}, u \restr{V_m})
    \le \wt \energy(u,u)
  \end{equation*}
  using~\eqref{eq:harmonic} and~\eqref{eq:harm.ext}.

  For the upper bound in~\eqref{eq:ew.direct2} we use the first
  inequality in~\eqref{eq:q-u-e.ew.result}.
\end{proof}


\subsection{Eigenvalue estimates for subsequent graphs}
\label{ssec:ev.graphs}
If the measure $\mu$ is self-similar (see \Subsec{ex.measure}), then
we can also show that the $k$-th eigenvalue of the graph energy forms
$\energy_m$ and $\energy_{m+1}$ are close to each other.  Denote by
$\lambda_k(G_m,\mu_m)$ the $k$-th eigenvalue of the graph Laplacian
$\Delta_m$ associated with $\energy_m$ (in increasing order, repeated
according to the multiplicity).  Blasiak, Strichartz and U{\u
  g}urcan~\cite[text after eq.~(2.9)]{bsu:08} believe that the
sequence $(\lambda_k(G_m,\mu_m))_{m \in \N}$ is monotonely increasing.
We can at least show the following:
\begin{proposition}
  \label{prp:ew.subseq}
  Assume that the measure $\mu$ on $K$ is self-similar, then we have
  \begin{equation*}
    \lambda_k(G_m,\mu_m)
    \le
    \frac 1{1-\dfrac{\lambda_{k,+}}{\lambda_1^\Dir}\cdot (\mu_+r_-)^m}
    \cdot \lambda_k(G_{m+1},\mu_{m+1})
  \end{equation*}
  where $\lambda_1^\Dir>0$ is the smallest eigenvalue of the Dirichlet
  Laplacian (the Laplacian $\Delta_{(G_1,\mu_1)}$ restricted to $V_1\setminus
  V_0$) and where $\lambda_{k,+}$ is an upper bound on
  $\lambda_k(G_{m+1},\mu_{m+1})$.
\end{proposition}
\begin{proof}
  We apply \Prp{ew.comp} with
  \begin{equation*}
    \map{J=J_{m+1,m}}{\HS_{m+1}=\lsqr{V_{m+1},\mu_{m+1}}}
                    {\HS_m=\lsqr{V_m,\mu_m}},
    \qquad
    J u := u \restr {V_m},
  \end{equation*}
  then we have
  \begin{align}
    \nonumber
    \normsqr[\lsqr{V_m,\mu_m}]{Ju} - \normsqr[\lsqr{V_{m+1},\mu_{m+1}}] u
    &= \sum_{x \in V_m} \abssqr{u(x)}(\mu_m(x)-\mu_{m+1}(x))
    - \sum_{x \in V_{m+1}\setminus V_m} \abssqr{u(x)} \mu_{m+1}(x)\\
    \nonumber
    &\ge \normsqr[\lsqr{V_m,\mu_m}] u (1-\mu_-)
    - \normsqr[\lsqr{V_{m+1}\setminus V_m,\mu_{m+1}}] u\\
    \label{eq:subs.ev.norm}
    &\ge - \normsqr[\lsqr{V_{m+1}\setminus V_m,\mu_{m+1}}] u
  \end{align}
  as $\mu_{m+1}(x) \ge \mu_- \mu_m(x)$ for $x \in V_m$
  (using~\eqref{eq:est.mu}).

  We can consider $\Delta_{(G_{m+1},\mu_{m+1})}$ restricted to
  $\lsqr{V_{m+1}\setminus V_m,\mu_{m+1}}$ as \emph{Dirichlet
    Laplacian} (denoted by $\Delta_{(G_{m+1},\mu_{m+1})}^\Dir$) with
  Dirichlet boundary condition on the ``boundary'' $V_m \subset
  V_{m+1}$ (see e.g.~\cite[Sec.~6.7]{post:16}).  We then have (using a
  variant of the min-max characterisation of eigenvalues)
  \begin{equation}
    \label{eq:subs.dir1}
    \lambda_1(\Delta_{(G_{m+1},\mu_{m+1})}^\Dir) 
    \le \frac{\energy_{m+1}(u)}
    {\normsqr[\lsqr{V_{m+1}\setminus V_m,\mu_{m+1}}] u}
  \end{equation}
  for any $\map u {V_{m+1}\setminus V_m} \C$.  Note that a Dirichlet
  eigenvalue is always positive.

  As the Dirichlet Laplacian on $V_{m+1}\setminus V_m$ is a direct sum
  of a rescaled copy of the one on  $V_1 \setminus V_0$, we can
  estimate the first eigenvalue on generation $m+1$ by
  \begin{align}
    \nonumber
    \lambda_1(\Delta_{(G_{m+1},\mu_{m+1})}^\Dir) 
    &= \inf_u \frac{\energy_{m+1}(u)}
      {\normsqr[\lsqr{V_{m+1}\setminus V_m,\mu_{m+1}}] u}\\
    \label{eq:eq:subs.dir2}
    &\ge \frac 1{(r_+\mu_-)^m}\inf_{u_1} \frac{\energy_1(u)}
      {\normsqr[\lsqr{V_1\setminus V_0,\mu_1}] u}
    = \frac 1{(r_+\mu_-)^m} \; 
    \overbrace{\lambda_1(\Delta_{(G_1,\mu_1)}^\Dir)}^{=:\lambda_1^\Dir}
  \end{align}
  using~\eqref{eq:self-similar2}, \eqref{eq:tau.est}
  and~\eqref{eq:self-similar.meas} (as $\mu_-$ is the minimum of the
  measure scaling factors $\mu_j$).  We conclude
  from~\eqref{eq:subs.ev.norm} and the Dirichlet eigenvalue
  estimates~\eqref{eq:subs.dir1}--\eqref{eq:eq:subs.dir2} that
  \begin{equation*}
    \normsqr[\lsqr{V_m,\mu_m}]{Ju} - \normsqr[\lsqr{V_{m+1},\mu_{m+1}}] u
    \ge - \frac 1 {\lambda_1(\Delta_{(G_{m+1},\mu_{m+1})}^\Dir)} \energy_{m+1}(u)
    \ge - \frac{(r_+\mu_-)^m}{\lambda_1^\Dir} \energy_{m+1}(u),
  \end{equation*}
  and we can choose $\delta_1=(r_+\mu_-)^m/\lambda_1^\Dir$ in
  \Prp{ew.comp}.  As $\energy_m(u \restr {V_m}) \le \energy_{m+1}(u)$
  by the compatibility of $(\energy_{G_m})_m$ the condition on the
  energy forms in \Prp{ew.comp} is also fulfilled.  The result now
  follows from \Prp{ew.comp}.
\end{proof}

\begin{remark}
  \label{rem:ew.subseq}
  If the fractal is symmetric with symmetric measure (see
  \Subsec{ex.symmetric}), then $(\mu_+r_-)^m=(r_0/N)^m$, and the
  factor is of the form $1+\Err((r_0/N)^m)$, depending on the a priori
  bound $\lambda_{k,+}$ on $\lambda_k(G_{m+1},\mu_{m+1})$, hence the
  sequence $(\lambda_k(G_m,\mu_m))_m$ is close to a monotonely
  increasing sequence at least for large $m$.
\end{remark}

%
\section{Examples}
\label{sec:examples}
%

In this section we present some classes of examples and some concrete
ones.  Note that the contraction ratio $\theta_j$ does not play any
role in the convergence results.
\subsection{Self-similar and other measures on \pcf fractals}
\label{ssec:ex.measure}

A probability measure $\mu$ on $K$ is called \emph{self-similar}, if
$\mu$ is a Borel regular probability measure and if there are
so-called \emph{measure scaling parameters} $\mu_1,\dots,\mu_N>0$ such
that
\begin{equation}
  \label{eq:self-similar.meas}
  \mu(A)= \sum_{j=1}^N \mu_j \mu(F_j^{-1}(A))
\end{equation}
for all Borel sets $A \subset K$.  Note that a self-similar measure
always exists for any fractal $K$ defined via a (finite) IFS.  In
particular, an $m$-cell $K_w=F_w(K)$ has measure
\begin{equation*}
  \mu(K_w) = \mu_{w_1} \cdot \ldots \cdot \mu_{w_m},
\end{equation*}
hence
\begin{equation}
  \label{eq:est.mu}
  \mu_{m,-} = (\mu_-)^m,
  \qquadtext{where}
  \mu_- := \min\{\mu_1,\dots,\mu_N\} \in (0,1).
\end{equation}

One choice of measure is given by
\begin{equation*}
  \mu_j=r_j^d,
\end{equation*}
where $d$ is defined by $\sum_{j=1}^N r_j^d=1$.  Here, $d$ is the
Hausdorff dimension of $(K,\resistMet)$,
cf.~\cite[Thm.~4.2.1]{kigami:01}, where $\resistMet$ is the
\emph{resistance metric}.  If the
embedding space $X=\R^d$ is the Euclidean space, then we can also
calculate the Hausdorff dimension $d_\eucl$ of $K$ with respect to the
\emph{Euclidean} metric: the number $d_\eucl$ is defined via $\sum_{j=1}^N
\theta_j^{d_\eucl} = 1$ for self-similar sets satisfying the so-called
\emph{open set condition} (see e.g.~\cite[Sec.~1.5]{kigami:01}). Here,
$\theta_j$ is the contraction ratio of the contractive similarity
$\map{F_j}{\R^d}{\R^d}$.

There is another natural measure on a fractal approximable by finite
weighted graphs, namely the \emph{Kusuoka energy measure}.  For a
definition, see e.g.~\cite[Def.~3.4]{teplyaev:08}.  A good reference
and a link to the unusual properties of the domain of the Laplacian on
a fractal can be found in~\cite{bst:99}.  There is even an explicit
formula (cf.~\cite[Thm.~6.1]{teplyaev:08}) for the Laplacian $\Delta_
{(K,\mu)}$, if $\mu$ is the Kusuoka energy measure resembling the
Laplacian of a Riemannian manifold (normalised to
$\mu(K)=1$). \Thm{main} applies also to this setting; one just needs
to calculate $\mu_m(x)=\int_K \psi_{x,m} \dd \mu$ with respect to this
measure.

\subsection{Symmetric \pcf fractals}
\label{ssec:ex.symmetric}

\begin{definition}
  \label{def:frac.symm}
  \indent
  \begin{enumerate}
  \item  
    \label{frac.symm.a}
    We call a self-similar measure $\mu$ \emph{symmetric} or
    \emph{homogeneous} if the measure rescaling parameters are all the
    same, namely
    \begin{equation*}
      \mu_j=\frac 1N
    \end{equation*}
    for $j=1,\dots, N$.
    
  \item 
    \label{frac.symm.b}
    We say that a fractal $K$ approximable by finite weighted graphs
    is \emph{symmetric}, if all contraction ratios $\theta_j$ of the
    IFS are the same and if all energy renormalisation parameters
    $r_j$ of the energy forms are the same, i.e., there exists
    $\theta_0,r_0 \in (0,1)$ such that
    \begin{equation*}
      \theta_0=\theta_j
      \qquadtext{and}
      r_0=r_j, \qquad
    \end{equation*}
    for all $j=1,\dots, N$.
  \item 
    \label{frac.symm.c}
    We say that $(K,\mu)$ is symmetric if $K$ is a symmetric fractal
    together with a symmetric self-similar measure $\mu$.

  \item 
    \label{frac.symm.d}
    Let $V_0 \subset K$ be the boundary of a fractal $K$ approximable
    by finite weighted graphs.  We say that $V_0$ is \emph{symmetric},
    if the weights $\mu_0(x_0)=\int_K \psi_{x_0,0} \dd \mu$ are all
    the same, namely
    \begin{equation*}
      \mu_0(x_0)=\frac 1{N_0}
    \end{equation*}
    for all $x_0 \in V_0$ (recall that $N_0:=\card {V_0} \le N$ and
    that $\sum_{x_0 \in V_0} \mu_0(x_0)=\mu(K)=1$).
  \end{enumerate}
\end{definition}

For a symmetric fractal $K$, the Hausdorff dimension of $K$ with
respect to the resistance metric is given by $d=-\log N/\log r_0>0$,
see e.g.~\cite[Thm.~4.2.1]{kigami:01}.  Moreover, the Hausdorff
dimension of $K$ with respect to the induced metric from the Euclidean
ambient space is $d_\eucl= -\log N/\log \theta_0>0$
(see~\cite[Sec.~1.5]{kigami:01}).

Recall that $W_{x,m}=\set{w \in W_m}{x \in K_w}$ denotes the set of
words $w$ such that the corresponding cells $K_w$ meet in the vertex
$x \in V_m$.  Denote by
\begin{equation*}
  N_1:= \max_{x \in V_m} \card{W_{x,m}}
\end{equation*}
the maximal number of cells meeting in one vertex (note that we have
$N_1 \le N_0 \le N$).  We can now estimate the weights $\mu_m(x)$ and
$\mu_{+,m}$:
\begin{lemma}
  \label{lem:symm.frac}
  Let $K$ be a symmetric fractal and $\mu$ be a symmetric self-similar
  measure.
  \begin{subequations}
    \begin{enumerate}
    \item
      \label{symm.frac.a}
      All $m$-cells $K_w$ have the same measure, namely
      \begin{equation}
        \label{eq:cell.m.symm}
        \mu(K_w) = \frac 1{N^m}
      \end{equation}
      for all $w \in W_m$.
      
    \item
      \label{symm.frac.b}
      The vertex measure $\mu_m(x)$ fulfils
      \begin{equation}
        \label{eq:mu.m.symm}
        \mu_m(x) \le \frac {\card{W_{x,m}}} {N^m}
        \le \frac{N_1}{N^m} \le \frac 1{N^{m-1}}.
      \end{equation}
      
    \item
      \label{symm.frac.c}
      If, in addition, the boundary $V_0$ of $K$ is symmetric, then we
      have
      \begin{equation}
        \label{eq:mu.m.symm'}
        \mu_m(x)=\frac {\card{W_{x,m}}} {N_0 N^m}
        \le \frac{N_1}{N_0 N^m} \le \frac 1{N^m}.
      \end{equation}
    \end{enumerate}
  \end{subequations}
\end{lemma}
\begin{proof}
  Eq.~\eqref{eq:cell.m.symm} follows from the scaling property
  ~\eqref{eq:self-similar.meas} $\mu$ with $\mu_j=1/N$.
  Eq.~\eqref{eq:mu.m.symm} is a direct consequence using the estimate
  $\psi_{x,m} \le 1$ and the fact that $\supp \psi_{x,m}=\bigcup_{w
    \in W_{x,m}} K_w$.  Finally, \eqref{eq:mu.m.symm'} follows again
  by the scaling property~\eqref{eq:self-similar.meas}.
\end{proof}

We immediately conclude together with \Thm{main}:
\begin{corollary}
  \label{cor:main'}
  Assume that $K$ is a symmetric fractal with symmetric self-similar
  measure $\mu$, then the $\delta_m$-quasi-unitary equivalence of the
  fractal energy form $\energy_K$ and the approximating graph energy
  forms $\energy_m=\energy_{G_m}$ as in \Thm{main} holds with
  \begin{equation*}
    \delta_m
    =\frac{(1+\sqrt {N_0})\sqrt{N_1}}{\sqrt{\conductance_{-,0}}} \cdot 
    \Bigl(\frac {r_0} N\Bigr)^{m/2}.
  \end{equation*}
  If in addition, the boundary $V_0$ of $K$ is also symmetric, then
  \Thm{main} holds with
  \begin{equation*}
    \delta_m
    =\frac{(1+\sqrt {N_0})\sqrt{N_1}}{\sqrt{\conductance_{-,0} N_0}} \cdot 
    \Bigl(\frac {r_0} N\Bigr)^{m/2}.
  \end{equation*}
\end{corollary}
\begin{proof}
  Recall the definition of $\mu_{+,m}$ in~\eqref{eq:mu.m.pm}.  From
  \Lem{symm.frac} we conclude
  \begin{equation*}
    \mu_{+,m} 
    = \frac{N_1}{N^m} 
    \le \frac 1{N^{m-1}}
    \quadtext{resp.}
    \mu_{+,m} 
    = \frac{N_1}{N_0N^m} 
    \le \frac 1{N^m}
  \end{equation*}
  in the case of a non-symmetric resp.\ symmetric boundary, hence the
  result follows from the expression of $\delta_m$
  in~\eqref{eq:main.delta} for \Thm{main}.
\end{proof}

For a symmetric fractal $K$, the number of vertices $\card {V_m}$ is
given recursively by $\card{V_0}=N_0$ and $\card{V_{m+1}}:=N
\card{V_m} - b$, where $b$ is the number of vertices identified from
one generation to the next one.  In particular, we have
\begin{align*}
  \card{V_0}=N_0, \qquad
  \card{V_m}=N^m N_0 - \frac{N^m-1}{N-1} b,
\end{align*}
and $\card {V_m}=N^m N_0 + \Err(N^{m-1})$ if $m \to \infty$, i.e., the
size of the approximating matrices of $\Delta_{(G_m,\mu_m)}$ increases
of order $N^m N_0$ as $m \to \infty$.

\subsection{\Sierpinski gasket and related fractals}
\label{ssec:sierpinski}

\myparagraph{The unit interval.}
The unit interval $K=[0,1]$ can be seen as a self-similar fractal with
$F_1(x)=x/2$ and $F_2(x)=x/2+1/2$ with boundary $V_0=\{0,1\}$.  This
fractal is symmetric, and if we choose the symmetric self-similar
measure $\mu$ (which is here the one-dimensional Lebesgue measure).
In particular, $K$, $\mu$ and $V_0$ are all symmetric and we have
\begin{align*}
  N&=N_0=N_1=2, &
  \theta_0&=\frac 12,&
  r_0&=\frac 12,&
  b&=1,&
  \card {V_m}&=2^m+1.
\end{align*}
Moreover, the conductances at generation $0$ are
$\conductance_{e_0}=1$ for the single edge $e_0$ in $G_0$.  The
approximating graphs $G_m$ are path graphs with $2^m+1$ vertices.  The
error is then given by
\begin{equation*}
  \delta_m
  =(1+\sqrt 2)\cdot 
  \frac 1{2^m}.
\end{equation*}
This convergence rate is quite good, as $\delta_m$ is smaller than
$0.01$ for $m \ge 8$, and the number of vertices is still not too
large, namely $\card {V_8}=2^8+1=257$.

\myparagraph{\Sierpinski gasket.}  The \Sierpinski gasket is given by
three contractions with ratio $\theta_0=1/2$ and fixed points at the
vertices of an equilateral triangle.  Here, $N=N_0=3$ and $N_1=2$.
Again, this fractal is symmetric and we fix the symmetric measure
$\mu$; also the boundary is symmetric.  In particular, we have
\begin{align*}
  N&=N_0=3, &
  N_1&=2, &
  \theta_0&=\frac 12,&
  r_0&=\frac 35,&
  b&=3,&
  \card {V_m}&
  =\frac32\left( 3^m +1\right).
\end{align*}
Moreover, we have $\conductance_{e_0}=1$; and the error is given by
\begin{equation*}
  \delta_m
  =\frac{(1+\sqrt 3)\sqrt 2}{\sqrt 3} \cdot 
  \frac 1 {5^{m/2}}.
\end{equation*}
The error $\delta_m$ is smaller than $0.01$ for $m \ge 7$, but the
number of vertices is then already quite large, namely $\card
{V_7}=3282$.

\myparagraph{\Sierpinski gaskets in higher dimension.}  Here, we
consider the self-similar symmetric set in $\R^{N-1}$ for $N \ge 2$
with contraction ratio $\theta_0=1/2$ and $N$ fixed points lying on an
$N$-dimensional pyramid with side length $1$.  For $N=2$, this is the
interval, for $N=3$ the fractal is the \Sierpinski gasket and for
$N=4$ the \Sierpinski pyramid.  We also have $N_0=N$ and $N_1=2$ (at
most two cells meet).  Also the boundary is symmetric.  We have here
\begin{align*}
  N&=N_0, &
  N_1&=2, &
  \theta_0&=\frac 12,&
  r_0&=\frac N{N+2},&
  b&=\frac {N(N-1)}2,&
  \card {V_m}&
  &=\frac N2 \left(N^m + 1\right).
\end{align*}
Moreover, we have $\conductance_{e_0}=1$; and the error is given by
\begin{equation*}
  \delta_m
  =\frac{(1+\sqrt N)\sqrt 2}{\sqrt N} \cdot \frac 1{(N+2)^{m/2}}.
\end{equation*}
If e.g.\ $N=4$, we have $\delta_m$ smaller than $0.01$ for $m \ge 6$,
but the number of vertices is then already quite large, namely $\card
{V_6}=8194$.

\myparagraph{Related fractals.} The \emph{level $n$ \Sierpinski
  gasket} $SG_n$ is a self-similar compact set in $\R^2$ defined as
follows (see e.g.~\cite[Ex.~4.1.1, Fig.~4.1.1]{strichartz:06}):
subdivide an equilateral triangle into $3(n-1)$ equilateral triangles
of side length $\theta_0=\theta_j=1/n$ of the original side length.
Each of these $3(n-1)$ triangles can be obtained by a similarity from
the original triangle with contraction ratio $1/n$ and a fixed point.
Hence, we have $N=3(n-1)$ fixed points, and only the $3$ fixed points
on the vertices of the original triangle form the boundary, hence
$N_0=3$.  By $D_3$-symmetry, we choose all conductances
$\conductance_{e_0}$ to be equal, say $1$.  The renormalisation
factors are again all the same and given e.g.\ for $n=3$ by
$r_0=\nicefrac7{15}$ (see \cite[Ex.~4.3.1]{strichartz:06}).  The
level $n$ \Sierpinski gasket is an example of a fractal where up to
$3$ cells meet at a junction point (if $n=3$, it is the junction point
in the centre of a cell in the previous generation), hence $N_1=3$.
For the level $3$ \Sierpinski gasket, we have
\begin{align*}
  N&=6, &
  N_0&=N_1=3, &
  \theta_0&=\frac 12,&
  r_0&=\frac 7{15},&
  b&=8,&
  \card {V_m}
  &=\frac75 \cdot 6^m +\frac 85.
\end{align*}
Moreover, we can choose $\conductance_{e_0}=1$ for all three edges of
$G_0$; and the error is given by
\begin{equation*}
  \delta_m
  =\frac{(1+\sqrt 3)\sqrt 3}{\sqrt 3}
  \left(\frac {\nicefrac 7{15}}6\right)^{m/2}
  =(1+\sqrt 3)
  \left(\frac 7{90}\right)^{m/2}.
\end{equation*}
The error $\delta_m$ is smaller than $0.01$ for $m \ge 5$, but the
number of vertices is then already quite large, namely $\card
{V_5}=10888$.

\subsection{Pentagasket}
\label{ssec:pentagasket}

The pentagasket is the self-similar structure with five similarities
of contraction ratio $\theta_0=(3-\sqrt 5)/2\approx 0.382$ and fixed
points located at the vertices of a pentagon (see~\cite{asst:01}
or~\cite[Example~4.3.3 and Exercise~4.3.2]{strichartz:06}).

One can start with all five vertices as $V_0$, hence $G_0$ is the
complete graph $K_5$ with $\card{E_0}=10$ edges (not embeddable in the
plane), cf.~\Fig{pentagasket5}.  Here, we have $N=N_0=5$, and at most
$N_1=2$ cells meet in one vertex.  Moreover,
\begin{align*}
  \theta_0&=\frac {3-\sqrt 5}2 \approx 0.382,&
  r_0&=\frac{\sqrt{161}-9}8 \approx 0.461,&
  b&=5,&
  \card {V_m}
  &=\frac{15}4 \cdot 5^m +\frac 54.
\end{align*}
Moreover, the conductances of the $10$ edges at generation $0$ can be
chosen to be
\begin{equation*}
  \conductance_{e_0}=
  \begin{cases}
    \dfrac{\sqrt{161}-7}{16}\approx 0.356 
         &\text{if $e_0$ is on the pentagon}\\[2ex]
    \dfrac{15-\sqrt{161}}{16}\approx 0.144 
         &\text{if $e_0$ joins every second vertex}
  \end{cases}
\end{equation*}
By symmetry, the five elementary harmonic functions on $K$ all have
the same integral value $\int_K \psi_{x_0,0} \dd \mu =1/5$ for $x_0
\in V_0$, hence $\mu_m(x)=2/5^{m+1}$ if $x$ is a junction point and
$\mu_m(x)=1/5^{m+1}$ if not.  Moreover, any $m$-cells has measure
$1/5^m$, hence $\mu_{+,m}=1/5^m$.

The error is then given by
\begin{equation*}
  \delta_m
  = \frac {(1+\sqrt 5)\sqrt 2}{(\conductance_{-,0})^{1/2}\sqrt 5}
  \cdot \Bigl(\frac{\sqrt{161}-9}{5 \cdot 8}\Bigr)^{m/2}
  \le 5.3848 \cdot 0.3037^m
\end{equation*}
see~\eqref{eq:main.delta}.  Here,
$q:=((\sqrt{161}-9)/40)^{1/2}=0.30366 \ldots \le 0.3037$; so $q^m$ is
smaller than $0.01$ for $m \ge 6$, and the number of vertices is
then $\card{V_6}=58595$.

\subsubsection*{A fractal with non-symmetric boundary}
If we start with only three boundary points $V_0$ (two opposite of the
third on the pentagon), and if one uses the five similarities first
contracting towards the fixed point and then rotating around the
centre of the pentagon by $2\pi/5$, then one obtains another
\begin{figure}[h] 
  \centering
  \includegraphics[width=0.8\textwidth]{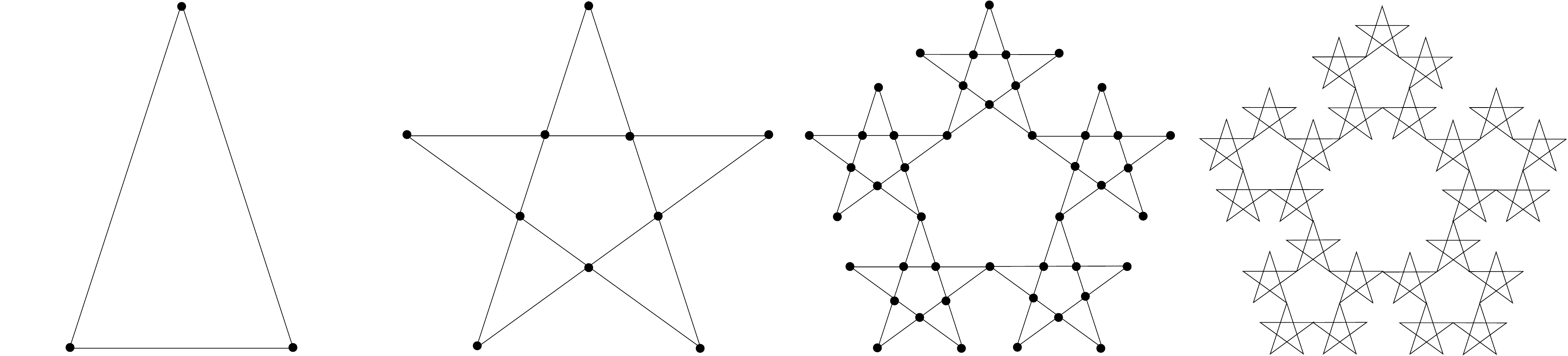}
  \caption{The pentagasket approximation graphs $G_m$ starting with
    only three boundary vertices for $m=0,\dots,3$ with $(\card
    {V_m})_m=(3, 10, 10\cdot 5-5=45, 45\cdot 5-5=220, \dots)$.}
  \label{fig:pentagasket3}
\end{figure}
approximating sequence for the pentagasket, see \Fig{pentagasket3}.
Here
\begin{equation*}
  \conductance_{e_0} =
  \begin{cases}
    \dfrac{\sqrt{161}-7}{14}\approx 0.406
         &\text{if $e_0$ is the base (shorter side) of $G_0$}\\[2ex]
    \dfrac{21-\sqrt{161}}{28}\approx 0.297 
         &\text{if $e_0$ is one of the sides of $G_0$}.
  \end{cases}
\end{equation*}
Note that $G_0$ is an isosceles triangle, hence the fractal has no
longer a symmetric boundary (as the rotational symmetry of $G_0$ is no
longer present).  The renormalisation factor $r_0=(\sqrt{161}-9)/8$ is
the same as above, only $N_0=3$ and
$\card{V_m} 
=\frac74 \cdot 5^m +\frac 54$ changes from the above setting, hence
the error term $\delta_m$ is the same as above, just with the factor
$(1+\sqrt 5)\sqrt 2/((\conductance_{-,0})^{1/2}\sqrt 5) \le 5.3848$
replaced by $(1+\sqrt 3)\sqrt 2/(\conductance_{-,0})^{1/2} \le
7.0917$.  Moreover, $\delta_m$ is smaller than $0.01$ if $m \ge 6$,
and here $\card{V_6}=27345$.

%
%

\newcommand{\etalchar}[1]{$^{#1}$}
\def\cprime{$'$}
\providecommand{\bysame}{\leavevmode\hbox to3em{\hrulefill}\thinspace}
\providecommand{\MR}{\relax\ifhmode\unskip\space\fi MR }
\providecommand{\MRhref}[2]{%
  \href{http://www.ams.org/mathscinet-getitem?mr=#1}{#2}
}
\providecommand{\href}[2]{#2}



\end{document}